\documentclass[12pt]{amsart}

\usepackage{amssymb}
\usepackage{amsmath}
\usepackage{amsthm}
\usepackage{graphics}

\allowdisplaybreaks

\numberwithin{equation}{section}

\theoremstyle{plain}
\newtheorem{thm}{Theorem}[section]
\newtheorem{prop}[thm]{Proposition}
\newtheorem{lem}[thm]{Lemma}
\newtheorem{cor}[thm]{Corollary}

\theoremstyle{definition}

\newtheorem{rem}[thm]{Remark}

\newcommand{\ichi}{\mathbf{1}}
\newcommand{\R}{\mathbb{R}}

\newcommand{\Z}{\mathbb{Z}}
\newcommand{\N}{\mathbb{N}}

\newcommand{\calA}{\mathcal{A}}

\newcommand{\calF}{\mathcal{F}}

\newcommand{\calS}{\mathcal{S}}

\begin{document}

\title[Limited smoothness conditions 
for bilinear Fourier multipliers]
{Limited smoothness conditions with 
mixed norms for bilinear Fourier multipliers}

\author[A. Miyachi]{Akihiko Miyachi}
\author[N. Shida]{Naoto Shida}
\author[N. Tomita]{Naohito Tomita}

\address[A. Miyachi]
{Department of Mathematics, 
Tokyo Woman's Christian University, 
Zempukuji, Suginami-ku, Tokyo 167-8585, Japan}
\address[N. Shida and N. Tomita]
{Department of Mathematics, 
Graduate School of Science, Osaka University, 
Toyonaka, Osaka 560-0043, Japan}

\email[A. Miyachi]{miyachi@lab.twcu.ac.jp}
\email[N. Shida]{u331453f@ecs.osaka-u.ac.jp}
\email[N. Tomita]{tomita@math.sci.osaka-u.ac.jp}

\date{\today}

\keywords{Bilinear Fourier multipliers,
H\"ormander multiplier theorem, $BMO$, Hardy spaces}

\thanks{This work was partially supported
by JSPS KAKENHI Grant Numbers 
JP16H03943 (Miyachi) and JP16K05201 (Tomita).}

\subjclass[2010]{42B15, 42B25, 42B30}

\begin{abstract}
In this paper,
the $L^2 \times L^{\infty} \to L^2$ and
$L^2 \times L^2 \to L^1$ boundedness
of bilinear Fourier multiplier operators is discussed
under weak smoothness conditions on multipliers.
As an application,
we prove the $L^2 \times BMO \to L^2$ and
$L^2 \times L^2 \to H^1$ boundedness
of bilinear operators with
multipliers of limited smoothness satisfying vanishing conditions.
\end{abstract}

\maketitle

\section{Introduction}\label{section1}
For $m(\xi_1,\xi_2) \in L^{\infty}(\R^n \times \R^n)$,
the bilinear Fourier multiplier operator $T_m$ is defined by
\[
T_m(f_1,f_2)(x)
=\frac{1}{(2\pi)^{2n}}
\int_{(\R^n)^2}
e^{ix \cdot (\xi_1+\xi_2)}
m(\xi_1,\xi_2)\widehat{f_1}(\xi_1)\widehat{f_2}(\xi_2)\,
d\xi_1 d\xi_2
\]
for $f_1,f_2 \in \calS(\R^n)$.
In the framework of multipliers which are smooth away from the origin,
it is well known that if $m$ satisfies
\begin{equation}\label{CM-multiplier}
|\partial^{\alpha_1}_{\xi_1}\partial^{\alpha_2}_{\xi_2}
m(\xi_1,\xi_2)|
\le C_{\alpha_1,\alpha_2}(|\xi_1|+|\xi_2|)^{-(|\alpha_1|+|\alpha_2|)},
\quad (\xi_1,\xi_2) \neq (0,0),
\end{equation}
for sufficiently many multi-indices
$\alpha_1,\alpha_2 \in \N_0^n=\{0,1,2,\dots\}^n$,
then the corresponding bilinear operator $T_m$
is bounded from $L^{p_1} \times L^{p_2}$ to $L^p$
for $1 \le p_1,p_2 \le \infty$ satisfying
$1/p_1+1/p_2=1/p$,
where $L^p$ is replaced by the weak $L^p$ space
if $p_1=1$ or $p_2=1$,
and by $BMO$ if $p_1=p_2=\infty$.
These fundamental results were given by
Coifman-Meyer \cite{CM-1, CM-2},
Kenig-Stein \cite{Kenig-Stein},
and Grafakos-Torres \cite{Grafakos-Torres}.
In the last decade,
the research on bilinear (multilinear) multipliers
of limited smoothness has been developed by several authors;
here we mention
Tomita \cite{Tomita}, Grafakos-Si \cite{Grafakos-Si},
Grafakos-Miyachi-Tomita \cite{Grafakos-Miyachi-Tomita},
Miyachi-Tomita \cite{Miyachi-Tomita-RMI}, and Park \cite{Park}.

To explain the results of \cite{Grafakos-Miyachi-Tomita, Miyachi-Tomita-RMI},
we shall introduce some notations.
Let $X_1,X_2$, and $Y$ be function spaces on $\R^n$ 
equipped with (quasi-)norms $\|\cdot \|_{X_1}$, $\|\cdot \|_{X_2}$, 
and $\|\cdot \|_{Y}$, 
respectively.  
If there exists a constant $A$ such that 
\begin{equation}\label{boundedness-X_1X_2Y}
\|T_m(f_1,f_2)\|_{Y}
\le A \|f_1\|_{X_1} \|f_2\|_{X_2} 
\;\;
\text{for all}
\;\;
f_1\in \calS \cap X_1  
\;\;
\text{and}
\;\;
f_2\in \calS \cap X_2,  
\end{equation}
then, 
with a slight abuse of terminology, 
we say that 
$T_m$ is bounded from 
$X_1 \times X_2$ to $Y$.  
The smallest constant $A$ of 
\eqref{boundedness-X_1X_2Y} 
is denoted by 
$\|T_m\|_{X_1 \times X_2 \to Y}$. 
For $s_1,s_2 \in \R$ and $m \in \calS'(\R^n \times \R^n)$,
the product type Sobolev norm $\|m\|_{W^{(s_1,s_2)}}$
is defined by
\begin{equation}\label{product-Sobolev}
\begin{split}
\|m\|_{W^{(s_1,s_2)}}
&=\| \langle y_1 \rangle^{s_1}\langle y_2 \rangle^{s_2}
\widehat{m}(y_1,y_2)\|_{L^2(\R^n_{y_1} \times \R^n_{y_2})}
\\
&=(2\pi)^{n}\|\langle D_{\xi_1} \rangle^{s_1}
\langle D_{\xi_2} \rangle^{s_2}
m(\xi_1,\xi_2)\|_{L^2(\R^n_{\xi_1} \times \R^n_{\xi_2})},
\end{split}
\end{equation}
where $\langle y_i \rangle=(1+|y_i|^2)^{1/2}$, $i=1,2$, and
\[
\langle D_{\xi_1} \rangle^{s_1}
\langle D_{\xi_2} \rangle^{s_2}
m(\xi_1,\xi_2)
=\frac{1}{(2\pi)^{2n}}
\int_{(\R^n)^2}
e^{i (\xi_1\cdot y_1+\xi_2\cdot y_2)}
\langle y_1 \rangle^{s_1}
\langle y_2 \rangle^{s_2}
\widehat{m}(y_1,y_2)\,
dy_1dy_2 .
\]
Let $\Psi$ be a function in $\calS(\R^d)$ satisfying
\begin{equation}\label{LP-homogeneous}
\mathrm{supp}\, \Psi
\subset \{\xi \in \R^d \,:\, 1/2 \le |\xi| \le 2\},
\qquad
\sum_{k \in \Z}\Psi(\xi/2^k)=1, \ \xi \in \R^d \setminus \{0\}.
\end{equation}
For $m(\xi_1,\xi_2) \in L^{\infty}(\R^n \times \R^n)$
and $j \in \Z$,
we set
\begin{equation}\label{def-mj}
m_j(\xi_1,\xi_2)=m(2^j\xi_1,2^j\xi_2)\Psi(\xi_1,\xi_2),
\quad (\xi_1,\xi_2) \in \R^n \times \R^n,
\end{equation}
where $\Psi \in \calS(\R^{2n})$ is
as in \eqref{LP-homogeneous} with $d=2n$.
The results of \cite{Grafakos-Miyachi-Tomita, Miyachi-Tomita-RMI}
state that if $s_1,s_2>n/2$, $1 \le p_1,p_2,p \le \infty$, and $1/p_1+1/p_2=1/p$,
then
\begin{equation}\label{CJM-RMI}
\|T_m\|_{L^{p_1} \times L^{p_2} \to L^p}
\lesssim
\sup_{j \in \Z}\|m_j\|_{W^{(s_1,s_2)}},
\end{equation}
where $L^{p_i}$ is replaced by $H^1$ if $p_i=1$,
and $L^p$ is replaced by $BMO$ if $p_1=p_2=\infty$.
(See \cite{Miyachi-Tomita-RMI} for the result in the full range $0 < p_1, p_2, p \le \infty$.)

One of the purposes of this paper
is to find an $L^2 \times L^{\infty} \to L^2$ estimate
sharper than \eqref{CJM-RMI}.
It should be mentioned that
the $L^2 \times L^{\infty} \to L^2$ boundedness
is a starting point to prove the $L^{p_1} \times L^{p_2} \to L^p$ one
for general $p_1,p_2,p$
in many problems; see e.g., \cite{Miyachi-Tomita-RMI}.
Moreover, we want to estimate $\|T_m\|_{L^2 \times L^\infty \to L^2}$
by the quantity which allows us to use a duality argument.
To use duality, we need to treat the map,
\[
m(\xi_1,\xi_2) \mapsto m^{*2}(\xi_1,\xi_2)=m(\xi_1, -\xi_1-\xi_2) 
\]
(see Subsection \ref{section4.2}).
However, it is impossible to control this map by
the product type Sobolev norm, \eqref{product-Sobolev},
that is,
$\|m^{*2}\|_{W^{(s_1,s_2)}}$
cannot be estimated by $\|m\|_{W^{(s_1,s_2)}}$.

Instead of \eqref{product-Sobolev},
we introduce the norms
\begin{align*}
&\|m\|_{W^{(s_1,s_2)}_1}
=\left\|\left\|
\langle D_{\xi_1} \rangle^{s_1}
\langle D_{\xi_2} \rangle^{s_2}
m(\xi_1,\xi_2)\right\|_{L^2(\R^n_{\xi_1})}
\right\|_{L^{\infty}(\R^n_{\xi_2})},
\\
&\|m\|_{W^{(s_1,s_2)}_2}
=\left\|\left\|
\langle D_{\xi_1} \rangle^{s_1}
\langle D_{\xi_2} \rangle^{s_2}
m(\xi_1,\xi_2)\right\|_{L^2(\R^n_{\xi_2})}
\right\|_{L^{\infty}(\R^n_{\xi_1})},
\end{align*}
where $s_1,s_2 \in \R$.
Given $s_1>0$ and $s_2>n/2$,
we can take $\widetilde{s}_1$ and $\widetilde{s}_2$
satisfying $s_1>\widetilde{s}_1>0$, $s_2>\widetilde{s}_2>n/2$,
and $\widetilde{s}_1+\widetilde{s}_2<s_2$.
Then the estimate
\[
\|m^{*2}\|_{W^{(\widetilde{s}_1,\widetilde{s}_2)}_2}
\lesssim \|m\|_{W^{(s_1,s_2)}_2}
\]
holds (see Lemma \ref{change-variable}),
and the map $m \mapsto m^{*2}$ is controlled by
the norm $\|\cdot\|_{W^{(s_1,s_2)}_2}$
with $s_1>0$ and $s_2>n/2$ in this sense.
It should be also mentioned that
if $s_1>0$, $s_2>n/2$, and $\|m\|_{W^{(s_1,s_2)}_2}<\infty$,
then $m(\xi_1,\xi_2)$ belongs to $L^{\infty}$
(more precisely, $\|m\|_{L^{\infty}} \lesssim \|m\|_{W^{(s_1,s_2)}_2}$)
and can be modified on a set of zero measure so that
the resulting function is continuous
(see Remark \ref{bounded-continuous}).

The first main result of this paper reads as follows.

\begin{thm}\label{mainthm}
Let $s_1>0$ and $s_2>n/2$.
Then
\[
\|T_m\|_{L^2 \times L^\infty \to L^2}
+\|T_m\|_{L^2 \times L^2 \to L^1}
\lesssim
\sup_{j \in \Z}\|m_j\|_{W^{(s_1,s_2)}_2} .
\]
\end{thm}

By the Sobolev embedding theorem,
for each $s_1,s_2 \in \R$ and $\epsilon>0$,
\[
|\langle D_{\xi_1} \rangle^{s_1}
\langle D_{\xi_2} \rangle^{s_2}
m(\xi_1,\xi_2)|
\lesssim
\left\|
\langle D_{\xi_1} \rangle^{s_1+n/2+\epsilon}
\langle D_{\xi_2} \rangle^{s_2}
m(\xi_1,\xi_2)\right\|_{L^2(\R^n_{\xi_1})} .
\]
Hence,
if $s_1,s_2>n/2$, $\widetilde{s}_1>0$, and $\widetilde{s}_1+n/2<s_1$,
then
\begin{equation}\label{comparison}
\sup_{j \in \Z}\|m_j\|_{W^{(\widetilde{s}_1,s_2)}_2}
\lesssim
\sup_{j \in \Z}\|m_j\|_{W^{(s_1,s_2)}} .
\end{equation}
This means that Theorem \ref{mainthm} is an improvement of \eqref{CJM-RMI}
for the case $(p_1,p_2)=(2,\infty), (2,2)$.

By symmetry, we also have
\[
\|T_m\|_{L^\infty \times L^2 \to L^2}
+\|T_m\|_{L^2 \times L^2 \to L^1}
\lesssim
\sup_{j \in \Z}\|m_j\|_{W^{(s_1,s_2)}_1} ,
\]
where $s_1>n/2$ and $s_2>0$.
Combining this with Theorem \ref{mainthm},
we see that
if $s_1, \widetilde{s_2}>n/2$ and $\widetilde{s}_1, s_2>0$, then
\[
\|T_m\|_{L^2 \times L^2 \to L^1}
\lesssim
\min\left\{
\sup_{j \in \Z}\|m_j\|_{W^{(s_1,s_2)}_1}, \
\sup_{j \in \Z}\|m_j\|_{W^{(\widetilde{s}_1,\widetilde{s}_2)}_2}
\right\} .
\]
In \cite{Miyachi-Tomita-TMJ},
it was also shown that if $s_1,s_2>n/2$, then
\[
\|T_m\|_{L^2 \times L^2 \to L^1}
\lesssim
\max
\left\{
\sup_{j \in \Z}\|m_j\|_{W^{(s_1,0)}_1}, \
\sup_{j \in \Z}\|m_j\|_{W^{(0,s_2)}_2} \right\} .
\]
However, we cannot compare these two results.
A related result can be also found in
Grafakos-He-Honz\'ik \cite{Grafakos-He-Honzik}.

Coifman-Lions-Meyer-Semmes \cite[Theorem and Remark V.1]{CLMS}
proved that if
$m$ is a smooth multiplier satisfying \eqref{CM-multiplier}
and $m(\xi_1,\xi_2)=0$ for $(\xi_1,\xi_2) \in \R^n \times \R^n \setminus \{(0,0)\}$
with $\xi_1+\xi_2=0$,
then $T_m$ is bounded from $H^{p_1} \times H^{p_2}$ to $H^p$,
$p_1,p_2,p>n/(n+1)$, $1/p_1+1/p_2=1/p$.
This result was extended to the full range $0<p_1,p_2 \le \infty$ and $0<p \le 1$
by Grafakos-Nakamura-Nguyen-Sawano \cite{GNNS}
under the suitable vanishing conditions.
Another purpose of this paper
is to give related results in this direction for multipliers of limited smoothness,
and the second main result reads as follows.

\begin{thm}\label{mainthm2}
Let $s_1>0$ and $s_2>n/2$.
\begin{enumerate}
\item
If $m(\xi_1,0)=0$ for $\xi_1 \in \R^n \setminus \{0\}$, then
\[
\|T_m\|_{L^2 \times BMO \to L^2}
\lesssim
\sup_{j \in \Z}\|m_j\|_{W^{(s_1,s_2)}_2}.
\]
\item
If $m(\xi_1,\xi_2)=0$
for $(\xi_1,\xi_2) \in \R^n \times \R^n \setminus \{(0,0)\}$
satisfying $\xi_1+\xi_2=0$, then
\[
\|T_m\|_{L^2 \times L^2 \to H^1}
\lesssim
\sup_{j \in \Z}\|m_j\|_{W^{(s_1,s_2)}_2}.
\]
\end{enumerate}
\end{thm}

By \eqref {comparison},
we have the following 
as a result of Theorem \ref{mainthm2}.

\begin{cor}
Let $s_1, s_2>n/2$.
\begin{enumerate}
\item
If $m(\xi_1,0)=0$ for $\xi_1 \in \R^n \setminus \{0\}$, then
\[
\|T_m\|_{L^2 \times BMO \to L^2}
\lesssim
\sup_{j \in \Z}\|m_j\|_{W^{(s_1,s_2)}}.
\]
\item
If $m(\xi_1,\xi_2)=0$
for $(\xi_1,\xi_2) \in \R^n \times \R^n \setminus \{(0,0)\}$
satisfying $\xi_1+\xi_2=0$, then
\[
\|T_m\|_{L^2 \times L^2 \to H^1}
\lesssim
\sup_{j \in \Z}\|m_j\|_{W^{(s_1,s_2)}}.
\]
\end{enumerate}
\end{cor}

The contents of this paper are as follows.
In Section \ref{section2},
we recall some preliminary facts.
In Section \ref{section3},
we give basic properties of $W^{(s_1,s_2)}_i$-norms.
In Sections \ref{section4} and \ref{section5},
we prove Theorems \ref{mainthm} and \ref{mainthm2},
respectively.
In Appendix,
we give the proofs of the lemmas cited in Section \ref{section2}.

\section{Preliminaries}\label{section2}

For two nonnegative quantities $A$ and $B$,
the notation $A \lesssim B$ means that
$A \le CB$ for some unspecified constant $C>0$,
and $A \approx B$ means that
$A \lesssim B$ and $B \lesssim A$.
For $a \ge 0$,
the notation $[a]$ means the integer part of $a$.
We denote by $\ichi_S$ the characteristic function of a set $S$.

Let $\calS(\R^n)$ and $\calS'(\R^n)$ be the Schwartz space of
rapidly decreasing smooth functions on $\R^n$ and its dual,
the space of tempered distributions, respectively.
We define the Fourier transform $\calF f$
and the inverse Fourier transform $\calF^{-1}f$
of $f \in \calS(\R^n)$ by
\[
\calF f(\xi)
=\widehat{f}(\xi)
=\int_{\R^n}e^{-ix\cdot\xi} f(x)\, dx
\quad \text{and} \quad
\calF^{-1}f(x)
=\frac{1}{(2\pi)^n}
\int_{\R^n}e^{ix\cdot \xi} f(\xi)\, d\xi.
\]

For a function $\sigma(x,\xi)\in L^{\infty}(\R^n \times \R^n)$, 
we define the linear pseudo-differential operator $\sigma(X,D)$ by 
\[
\sigma(X,D)f(x)
=\frac{1}{(2\pi)^n}\int_{\R^n}e^{i x \cdot \xi}
\sigma(x,\xi)\widehat{f}(\xi)\, d\xi,
\qquad f \in \calS(\R^n).
\]
In particular, if $\sigma$ is an $x$-independent symbol,
then we denote by $\sigma(D)$ the
corresponding linear Fourier multiplier operator.
The Hardy-Littlewood maximal operator $M$ is defined by
\[
Mf(x)=\sup_{r>0}\frac{1}{r^n}
\int_{|x-y| < r}|f(y)|\, dy,
\]
where $f$ is a locally integrable function on $\R^n$.

Let $F(\xi_1,\xi_2)$ be a function on $\R^n \times \R^n$.
We denote
the $L^{p_2}_{\xi_2}(L^{p_1}_{\xi_1})$-norm
and $L^{p_1}_{\xi_1}(L^{p_2}_{\xi_2})$-norm
of $F(\xi_1,\xi_2)$ by
$\left\|\left\|F(\xi_1,\xi_2)
\right\|_{L^{p_1}_{\xi_1}}\right\|_{L^{p_2}_{\xi_2}}$
and
$\left\|\left\|F(\xi_1,\xi_2)
\right\|_{L^{p_2}_{\xi_2}}\right\|_{L^{p_1}_{\xi_1}}$,
\begin{align*}
&\left\|\left\|F(\xi_1,\xi_2)\right\|_{L^{p_1}_{\xi_1}}
\right\|_{L^{p_2}_{\xi_2}}
=\left\{\int_{\R^n}\left(\int_{\R^n}
|F(\xi_1,\xi_2)|^{p_1}\, d\xi_1 \right)^{p_2/p_1}\, d\xi_2 \right\}^{1/p_2},
\\
&\left\|\left\|F(\xi_1,\xi_2)\right\|_{L^{p_2}_{\xi_2}}
\right\|_{L^{p_1}_{\xi_1}}
=\left\{\int_{\R^n}\left(\int_{\R^n}
|F(\xi_1,\xi_2)|^{p_2}\, d\xi_2 \right)^{p_1/p_2}\, d\xi_1 \right\}^{1/p_1},
\end{align*}
with usual modifications if $p_1=\infty$ or $p_2=\infty$.
In the case $p_1=p_2$,
we simply write $\|\cdot\|_{L^{p_1}_{\xi_1,\xi_2}}$
instead of $\big\|\|\cdot\|_{L^{p_1}_{\xi_1}}\big\|_{L^{p_2}_{\xi_2}}$.
For $s \in \R$, the $L^2$-based Sobolev space $W^s(\R^n)$
consists of all $f \in \calS'(\R^n)$
such that
\[
\|f\|_{W^s}=\|\langle \cdot \rangle^s \widehat{f}\|_{L^2}<\infty.
\]

We recall the definitions and some properties of 
Hardy spaces $H^p$ and the space $BMO$ on $\R^n$
(see, e.g., \cite[Chapters 3 and 4]{Stein}).
Let $0<p \le \infty$, and let $\phi \in \calS(\R^n)$ be such that
$\int_{\R^n}\phi(x)\, dx \neq 0$. 
Then the Hardy space $H^p(\R^n)$ consists of all $f \in \calS'(\R^n)$
such that
\[
\|f\|_{H^p}=\Big\|\sup_{0<t<\infty}|[t^{-n}\phi(\cdot/t)]*f|
\Big\|_{L^p}<\infty.
\]
It is known that $H^p(\R^n)$ does not depend 
on the choice of the function $\phi$,
$H^1(\R^n)$ is continuously embedded into $L^1(\R^n)$,
and $H^p(\R^n)=L^p(\R^n)$, $1<p \le \infty$.
The space $BMO(\R^n)$ consists of
all locally integrable functions $f$ on $\R^n$ such that
\[
\|f\|_{BMO}
=\sup_{Q}\frac{1}{|Q|}
\int_{Q}|f(x)-f_Q|\, dx<\infty, 
\]
where $f_Q$ is the average of $f$ on $Q$
and the supremum is taken over all cubes $Q$ in $\R^n$.
It is known that the dual space of
$H^1(\R^n)$ is $BMO(\R^n)$.

The following two lemmas
are essentially the same as
\cite[Lemma 3.3]{Miyachi-Tomita-TMJ}
and \cite[Lemma 3.2]{Grafakos-Miyachi-Tomita},
but we give their proofs in Appendix for the reader's convenience.

\begin{lem}\label{TMJ-pointwise}
Let $s \ge 0$, and set $\zeta_j(x)=2^{jn}(1+2^j|x|)^{-2s}$, $j \in \Z$.
Then 
\begin{align*}
&|T_{m(2^{-j}\cdot)}(f_1,f_2)(x)|
\lesssim
(\zeta_j*|f_1|^2)(x)^{1/2}
\left\|\int_{\R^n}e^{ix\cdot\xi_2}
\langle D_{\xi_1}\rangle^{s}
m(\xi_1,2^{-j}\xi_2)\widehat{f_2}(\xi_2)\, d\xi_2
\right\|_{L^2_{\xi_1}},
\\
&|T_{m(2^{-j}\cdot)}(f_1,f_2)(x)|
\lesssim
(\zeta_j*|f_2|^2)(x)^{1/2}
\left\|\int_{\R^n}e^{ix\cdot\xi_1}
\langle D_{\xi_2}\rangle^{s}
m(2^{-j}\xi_1,\xi_2)\widehat{f_1}(\xi_1)\, d\xi_1
\right\|_{L^2_{\xi_2}} .
\end{align*}
\end{lem}

\begin{lem}\label{CJM-Carleson}
Let $\psi \in \calS(\R^n)$ be such that $\psi(0)=0$,
and $\zeta_j(x)=2^{jn}(1+2^j|x|)^{-(n+\epsilon)}$, $j \in \Z$,
with some $\epsilon>0$.
Then,
\[
\left(\sum_{j \in \Z}\int_{\R^n}
(\zeta_j*|f|)(x)^{2}(\zeta_j * |\psi(2^{-j}D)g|^2)(x)\,
dx \right)^{1/2}
\lesssim \|f\|_{L^2}\|g\|_{BMO}.
\]
\end{lem}

\section{Basic properties of $W^{(s_1,s_2)}_i$-norms}
\label{section3}

Let $\psi \in \calS(\R^n)$ be as in
\eqref{LP-homogeneous} with $d=n$,
and set $\psi_0(\xi)=1-\sum_{k=1}^{\infty}\psi(\xi/2^k)$
and $\psi_k(\xi)=\psi(\xi/2^k)$, $k \ge 1$.
Then
$\mathrm{supp}\, \psi_0
\subset \{|\xi| \le 2\}$,
$\mathrm{supp}\, \psi_k \subset \{2^{k-1} \le |\xi| \le 2^{k+1}\}$,
$k \ge 1$,
and $\sum_{k \in \N_0}\psi_k(\xi)=1, \ \xi \in \R^n$.
We denote by $\Delta_{\boldsymbol{k}}$,
$\boldsymbol{k}=(k_1,k_2) \in \N_0^2$,
the Littlewood-Paley operator of product type,
namely,
\begin{align*}
\Delta_{\boldsymbol{k}}m(\xi_1,\xi_2)
&=\psi_{k_1}(D_{\xi_1})\psi_{k_2}(D_{\xi_2})m(\xi_1,\xi_2)
\\
&=\frac{1}{(2\pi)^{2n}}
\int_{(\R^n)^2}
e^{i (\xi_1\cdot y_1+\xi_2\cdot y_2)}
\psi_{k_1}(y_1)\psi_{k_2}(y_2)
\widehat{m}(y_1,y_2)\,
dy_1dy_2.
\end{align*}
For $\boldsymbol{s}=(s_1, s_2) \in \R^2$,
we introduce the Besov type norms by
\begin{align*}
&\|m\|_{B^{(s_1,s_2)}_1}
=\sup_{\boldsymbol{k} \in \N_0^2}
2^{\boldsymbol{k} \cdot \boldsymbol{s}}
\left\|\left\|\Delta_{\boldsymbol{k}}m(\xi_1,\xi_2)
\right\|_{L^2(\R^n_{\xi_1})}\right\|_{L^{\infty}(\R^n_{\xi_2})},
\\
&\|m\|_{B^{(s_1,s_2)}_2}
=\sup_{\boldsymbol{k} \in \N_0^2}
2^{\boldsymbol{k} \cdot \boldsymbol{s}}
\left\|\left\|\Delta_{\boldsymbol{k}}m(\xi_1,\xi_2)
\right\|_{L^2(\R^n_{\xi_2})}\right\|_{L^{\infty}(\R^n_{\xi_1})},
\end{align*}
where $\boldsymbol{k}\cdot\boldsymbol{s}=k_1s_1+k_2s_2$.
The more general definition of Besov spaces of product type
can be found in Sugimoto \cite[Definition 1.2]{Sugimoto}.
Roughly speaking,
the following proposition says that
the two norms $\|\cdot\|_{W^{(s_1,s_2)}_i}$ and
$\|\cdot\|_{B^{(s_1,s_2)}_i}$ are essentially equivalent
(since $\widetilde{s}_i$ can be chosen arbitrarily close to $s_i$).

\begin{prop}\label{prop-equiv-norm}
Let $s_i, \widetilde{s}_i$
be real numbers satisfying $\widetilde{s}_i<s_i$, $i=1,2$.
Then
\begin{equation}\label{equiv-norm}
\|m\|_{B^{(\widetilde{s}_1,\widetilde{s}_2)}_i}
\lesssim
\|m\|_{W^{(\widetilde{s}_1,\widetilde{s}_2)}_i}
\lesssim
\|m\|_{B^{(s_1,s_2)}_i},
\quad i=1,2.
\end{equation}
\end{prop}
\begin{proof}
By symmetry, it is sufficient to consider the case $i=2$
in \eqref{equiv-norm}.
We first prove the latter inequality in \eqref{equiv-norm}.
Let $\widetilde{\psi}_0, \widetilde{\psi}$
be such that $\widetilde{\psi}_0=1$ on $\{|\xi| \le 2\}$,
$\mathrm{supp}\, \widetilde{\psi}_0 \subset \{|\xi| \le 4\}$,
$\widetilde{\psi}=1$ on $\{1/2 \le |\xi| \le 2\}$,
and
$\mathrm{supp}\, \widetilde{\psi} \subset \{1/4 \le |\xi| \le 4\}$.
Set $\widetilde{\psi}_k(\xi)=\widetilde{\psi}(2^{-k}\xi)$, $k \ge 1$,
and note that $\psi_{k}=\widetilde{\psi}_k \psi_k$, $k \ge 0$.
We write
$\langle D_{\xi_1} \rangle^{\widetilde{s}_1}
\langle D_{\xi_2} \rangle^{\widetilde{s}_2}m$ as
\begin{equation}\label{equiv-norm-(1)}
\begin{split}
&\langle D_{\xi_1} \rangle^{\widetilde{s}_1}
\langle D_{\xi_2} \rangle^{\widetilde{s}_2}
m(\xi_1,\xi_2)
=\sum_{\boldsymbol{k} \in \N_0^2}
\langle D_{\xi_1} \rangle^{\widetilde{s}_1}
\langle D_{\xi_2} \rangle^{\widetilde{s}_2}
\Delta_{\boldsymbol{k}}m(\xi_1,\xi_2)
\\
&=\sum_{\boldsymbol{k} \in \N_0^2}
\calF^{-1}_{(y_1,y_2) \to (\xi_1,\xi_2)}
\left[\langle y_1 \rangle^{\widetilde{s}_1}
\langle y_2 \rangle^{\widetilde{s}_2}
\widetilde{\psi}_{k_1}(y_1)\widetilde{\psi}_{k_2}(y_2)
\psi_{k_1}(y_1)\psi_{k_2}(y_2)\widehat{m}(y_1,y_2)\right]
\\
&=\sum_{\boldsymbol{k} \in \N_0^2}
(\calF^{-1}\tau_{\boldsymbol{k}}
^{(\widetilde{s}_1,\widetilde{s}_2)})
*(\Delta_{\boldsymbol{k}}m)(\xi_1,\xi_2)
\end{split}
\end{equation}
with
\[
\tau_{\boldsymbol{k}}
^{(\widetilde{s}_1,\widetilde{s}_2)}(y_1,y_2)
=\langle y_1 \rangle^{\widetilde{s}_1}
\langle y_2 \rangle^{\widetilde{s}_2}
\widetilde{\psi}_{k_1}(y_1)\widetilde{\psi}_{k_2}(y_2) .
\]
Since
$|\partial^{\alpha_i}_{y_i} \langle y_i \rangle^{\widetilde{s}_i}|
\lesssim \langle y_i \rangle^{\widetilde{s}_i-|\alpha_i|}$
and $\langle y_i \rangle \approx 2^{k_i}$
for $y_i \in \mathrm{supp}\, \widetilde{\psi}_{k_i}$,
we have
\[
|\partial^{\alpha_1}_{y_1}\partial^{\alpha_2}_{y_2}
\tau_{\boldsymbol{k}}
^{(\widetilde{s}_1,\widetilde{s}_2)}(y_1,y_2)|
\lesssim 2^{k_1(\widetilde{s}_1-|\alpha_1|)}
2^{k_2(\widetilde{s}_2-|\alpha_2|)}
\ichi_{\{|y_1| \lesssim 2^{k_1}, \, |y_2| \lesssim 2^{k_2}\}} .
\]
Then, by integration by parts,
\[
|\calF^{-1}\tau_{\boldsymbol{k}}
^{(\widetilde{s}_1,\widetilde{s}_2)}(\xi_1,\xi_2)|
\lesssim \prod_{i=1}^2 2^{k_i(\widetilde{s}_i+n)}(1+2^{k_i}|\xi_i|)^{-N_i},
\]
where $N_i$ is a positive integer
satisfying $N_i>n$ for $i=1,2$,
and consequently
\begin{equation}\label{equiv-norm-(3)}
\|\calF^{-1}\tau_{\boldsymbol{k}}
^{(\widetilde{s}_1,\widetilde{s}_2)}\|_{L^1}
\lesssim 2^{k_1\widetilde{s}_1+k_2\widetilde{s}_2}.
\end{equation}
Hence, it follows from \eqref{equiv-norm-(1)},
\eqref{equiv-norm-(3)}, and Young's inequality with mixed norm
(\cite[Part II, Theorem 1]{Benedek-Panzone}) that
\begin{align*}
&\left\|\left\|\langle D_{\xi_1}
\rangle^{\widetilde{s}_1}
\langle D_{\xi_2}
\rangle^{\widetilde{s}_2}m(\xi_1,\xi_2)
\right\|_{L^2_{\xi_2}}\right\|_{L^{\infty}_{\xi_1}}
\le \sum_{\boldsymbol{k} \in \N_0^2}
\|\calF^{-1}\tau_{\boldsymbol{k}}
^{(\widetilde{s}_1,\widetilde{s}_2)}\|_{L^1}
\left\|\left\|\Delta_{\boldsymbol{k}}m(\xi_1,\xi_2)
\right\|_{L^2_{\xi_2}}\right\|_{L^{\infty}_{\xi_1}}
\\
&\lesssim
\left(\sup_{\boldsymbol{k} \in \N_0^2}
2^{k_1s_1+k_2s_2}
\left\|\left\|\Delta_{\boldsymbol{k}}m(\xi_1,\xi_2)
\right\|_{L^2_{\xi_2}}\right\|_{L^{\infty}_{\xi_1}}\right)
\sum_{\boldsymbol{k} \in \N_0^2}
2^{k_1(\widetilde{s}_1-s_1)+k_2(\widetilde{s}_2-s_2)}.
\end{align*}
By the assumption $\widetilde{s}_i<s_i$, $i=1,2$,
this implies the latter inequality in \eqref{equiv-norm}.

We next consider the former inequality in \eqref{equiv-norm}.
The function $\Delta_{\boldsymbol{k}}m$ can be written as
\begin{align*}
\Delta_{\boldsymbol{k}}m(\xi_1,\xi_2)
&=\Delta_{\boldsymbol{k}}
\langle D_{\xi_1} \rangle^{-\widetilde{s}_1}
\langle D_{\xi_2} \rangle^{-\widetilde{s}_2}
\langle D_{\xi_1} \rangle^{\widetilde{s}_1}
\langle D_{\xi_2} \rangle^{\widetilde{s}_2}
m(\xi_1,\xi_2)
\\
&=(\calF^{-1}\sigma_{\boldsymbol{k}}^{(-\widetilde{s}_1,-\widetilde{s}_2)})*
(\langle D_{\xi_1} \rangle^{\widetilde{s}_1}
\langle D_{\xi_2} \rangle^{\widetilde{s}_2}m)(\xi_1,\xi_2)
\end{align*}
with
\[
\sigma_{\boldsymbol{k}}^{(-\widetilde{s}_1,-\widetilde{s}_2)}(y_1,y_2)
=\langle y_1 \rangle^{-\widetilde{s}_1}\langle y_2 \rangle^{-\widetilde{s}_2}
\psi_{k_1}(y_1)\psi_{k_2}(y_2) .
\]
Since \eqref{equiv-norm-(3)} has been proved
without any assumption on $\widetilde{s}_i$,
and  $\psi_{k_i}$ is essentially the same as $\widetilde{\psi}_{k_i}$,
we have
\[
\|\calF^{-1}\sigma_{\boldsymbol{k}}^{(-\widetilde{s}_1,-\widetilde{s}_2)}\|_{L^1}
\lesssim 2^{-(k_1\widetilde{s}_1+k_2\widetilde{s}_2)}.
\]
Therefore, by Young's inequality with mixed norm,
\[
\Big\|\|\Delta_{\boldsymbol{k}}m(\xi_1,\xi_2)
\|_{L^2_{\xi_2}}\Big\|_{L^{\infty}_{\xi_1}}
\lesssim
2^{-(k_1\widetilde{s}_1+k_2\widetilde{s}_2)}
\Big\|\|\langle D_{\xi_1} \rangle^{\widetilde{s}_1}
\langle D_{\xi_2} \rangle^{\widetilde{s}_2}m(\xi_1,\xi_2)
\|_{L^2_{\xi_2}}\Big\|_{L^{\infty}_{\xi_1}},
\]
which gives the former inequality in \eqref{equiv-norm}.
\end{proof}

\begin{rem}\label{rem-equiv-norm}
In the same way as in the proof of
Proposition \ref{prop-equiv-norm},
we can prove that
\begin{equation}\label{rem-equiv-norm-1}
\|m\|_{W^{(\widetilde{s}_1,\widetilde{s}_2)}_2}
\lesssim
\|m\|_{W^{(s_1,s_2)}_2}
\end{equation}
for each $\widetilde{s}_i \le s_i$, $i=1,2$,
and
\begin{equation}\label{rem-equiv-norm-2}
\sup_{k_1 \in \N_0}\left(
2^{k_1s_1}\left\|\left\|\psi_{k_1}(D_{\xi_1})
\langle D_{\xi_2}\rangle^{s_2}
m(\xi_1,\xi_2)\right\|_{L^2_{\xi_2}}
\right\|_{L^{\infty}_{\xi_1}}\right)
\lesssim
\|m\|_{W^{(s_1,s_2)}_2}
\end{equation}
for each $s_1, s_2 \in \R$.
(Similar assertions also hold for
the $W^{(s_1,s_2)}_1$-norm
by symmetry.)
In fact, for the former assertion,
instead of \eqref{equiv-norm-(1)} we write
\begin{align*}
\langle D_{\xi_1} \rangle^{\widetilde{s}_1}
\langle D_{\xi_2} \rangle^{\widetilde{s}_2}m
=\sum_{\boldsymbol{k} \in \N_0^2}
(\calF^{-1}\sigma_{\boldsymbol{k}}^{(\widetilde{s}_1-s_1,
\widetilde{s}_2-s_2)})
*(\langle D_{\xi_1} \rangle^{s_1}
\langle D_{\xi_2} \rangle^{s_2}m)
\end{align*}
with
$\sigma_{\boldsymbol{k}}^{(\widetilde{s}_1-s_1,
\widetilde{s}_2-s_2)}(y_1,y_2)=
\langle y_1 \rangle^{\widetilde{s}_1-s_1}
\langle y_2 \rangle^{\widetilde{s}_2-s_2}
\psi_{k_1}(y_1)\psi_{k_2}(y_2)$,
where we do not decompose
$\langle D_{\xi_1} \rangle^{\widetilde{s}_1}
\langle D_{\xi_2} \rangle^{\widetilde{s}_2}m(\xi_1,\xi_2)$
with respect to the $\xi_i$-variable
if $\widetilde{s}_i=s_i$,
and instead of \eqref{equiv-norm-(3)} we use
\[
\|\calF^{-1}\sigma_{\boldsymbol{k}}^{(\widetilde{s}_1-s_1,
\widetilde{s}_2-s_2)}\|_{L^1}
\lesssim 2^{k_1(\widetilde{s}_1-s_1)+k_2(\widetilde{s}_2-s_2)}.
\]
For the latter assertion,
we write
\[
\psi_{k_1}(D_{\xi_1})
\langle D_{\xi_2}\rangle^{s_2}m(\xi_1,\xi_2)
=\int_{\R^n}
\calF^{-1}\sigma_{k_1}^{-s_1}(\eta_1)
(\langle D_{\xi_1} \rangle^{s_1}
\langle D_{\xi_2} \rangle^{s_2}m)(\xi_1-\eta_1,\xi_2)\,
d\eta_1
\]
with $\sigma_{k_1}^{-s_1}(y_1)=\langle y_1 \rangle^{-s_1}\psi_{k_1}(y_1)$,
and use
$\|\calF^{-1}\sigma_{k_1}^{-s_1}\|_{L^1}
\lesssim 2^{-k_1s_1}$.
\end{rem}

Making a slight modification on the proof of
{\cite[Lemma 3.4]{Grafakos-Miyachi-Tomita}
or \cite[Lemma 3.1]{Miyachi-Tomita-TMJ}},
and using the fact of pointwise multipliers
in Besov spaces of product type
(\cite[Theorem 1.4]{Sugimoto}),
\begin{equation}\label{pointwise-multipliers-Besov}
\|\Phi m\|_{B^{(s_1,s_2)}_i}
\lesssim
\left(\sup_{\boldsymbol{k} \in \N_0^2}
2^{\boldsymbol{k} \cdot \boldsymbol{s}}
\left\|\Delta_{\boldsymbol{k}}\Phi(\xi_1,\xi_2)
\right\|_{L^\infty(\R^n_{\xi_1,\xi_2})}\right)
\|m\|_{B^{(s_1,s_2)}_i},
\quad i=1,2,
\end{equation}
where $s_i$ is a positive number for $i=1,2$,
we can prove the following.
\begin{lem}\label{remove-cut}
Let $\widetilde{\Psi} \in \calS(\R^{2n})$
be such that $\mathrm{supp}\, \widetilde{\Psi}$
is a compact set of $\R^{2n} \setminus \{0\}$.
Assume that $\Phi \in C^{\infty}(\R^{2n}\setminus\{0\})$
satisfies
\[
|\partial^{\alpha_1}_{\xi_1}\partial^{\alpha_2}_{\xi_2}\Phi(\xi_1,\xi_2)|
\le C_{\alpha_1,\alpha_2}
(|\xi_1|+|\xi_2|)^{-(|\alpha_1|+|\alpha_2|)},
\quad (\xi_1,\xi_2) \neq (0,0),
\]
for each $(\alpha_1,\alpha_2) \in (\N_0^n)^2$.
For $m \in L^{\infty}(\R^{2n})$ and $j \in \Z$,
set
\[
\widetilde{m}_j(\xi_1,\xi_2)
=m(2^j\xi_1,2^j\xi_2)\Phi(2^j\xi_1,2^j\xi_2)
\widetilde{\Psi}(\xi_1,\xi_2),
\]
and define $m_j$ by \eqref{def-mj}.
Then the following hold.
\begin{enumerate}
\item
For each $s_i>0$, $i=1,2$, 
\[
\sup_{j \in \Z}\|\widetilde{m}_j\|_{B^{(s_1,s_2)}_i}
\lesssim \sup_{j \in \Z}\|m_j\|_{B^{(s_1,s_2)}_i},
\quad i=1,2.
\]
\item
For each $s_i > \widetilde{s}_i \ge 0$, $i=1,2$,
\begin{equation}\label{remove-cut-2}
\sup_{j \in \Z}
\|\widetilde{m}_j\|_{W^{(\widetilde{s}_1,\widetilde{s}_2)}_i}
\lesssim \sup_{j \in \Z}
\|m_j\|_{W^{(s_1,s_2)}_i},
\quad i=1,2.
\end{equation}
\end{enumerate}
\end{lem}
\begin{proof}
We assume that
$\mathrm{supp}\, \widetilde{\Psi} \subset
\{2^{-k_0} \le |\xi| \le 2^{k_0}\}$
for some $k_0 \ge 1$,
where $\xi=(\xi_1,\xi_2) \in (\R^n)^2$.
Let $\Psi$ be as in \eqref{LP-homogeneous} with $d=2n$.
Then, since
$\mathrm{supp}\, \Psi(2^{-k}\xi) \subset
\{2^{k-1} \le |\xi| \le 2^{k+1}\}$,
\begin{align*}
\widetilde{m}_j(\xi)
=\sum_{k=-k_0}^{k_0}
\Phi(2^j\xi)\widetilde{\Psi}(\xi)
m(2^j\xi)\Psi(2^{-k}\xi)
=\sum_{k=-k_0}^{k_0}\Phi_j(\xi)m_{j+k}(2^{-k}\xi),
\end{align*}
where $m_{j+k}$ is defined by \eqref{def-mj}
and $\Phi_j(\xi)=\Phi(2^j\xi)\widetilde{\Psi}(\xi)$.
Hence, by \eqref{pointwise-multipliers-Besov},
\[
\|\widetilde{m}_j\|_{B^{(s_1,s_2)}_i}
\lesssim
\sum_{k=-k_0}^{k_0}
\left(\sup_{\boldsymbol{\ell} \in \N_0^2}
2^{\boldsymbol{\ell} \cdot \boldsymbol{s}}
\left\|\Delta_{\boldsymbol{\ell}}\Phi_j
\right\|_{L^\infty}\right)
\|m_{j+k}(2^{-k}\cdot)\|_{B^{(s_1,s_2)}_i}.
\]
Therefore, combining this with
\[
\sup_{\boldsymbol{\ell} \in \N_0^2}
2^{\boldsymbol{\ell} \cdot \boldsymbol{s}}
\left\|\Delta_{\boldsymbol{\ell}}\Phi_j
\right\|_{L^\infty}
\lesssim \sup_{j \in \Z}
\left(\max_{|\alpha_i| \le [s_i]+1, \, i=1,2}
\|\partial^{\alpha_1}_{\xi_1}
\partial^{\alpha_2}_{\xi_2}\Phi_j\|_{L^{\infty}}\right)
<\infty
\]
(see e.g., \cite[Theorem 2.3.8]{Triebel})
and
\begin{align*}
\|m_{j+k}(2^{-k}\cdot)\|_{B^{(s_1,s_2)}_i}
&\lesssim (2^{-k})^{-n/2}
\left(\max\{1,2^{-k}\}\right)^{s_1+s_2}
\|m_{j+k}\|_{B^{(s_1,s_2)}_i}
\\
&\lesssim
\sup_{j \in \Z}\|m_{j}\|_{B^{(s_1,s_2)}_i}
\end{align*}
for $|k| \le k_0$
(see \cite[Proposition 1.1]{Sugimoto} for the first inequality),
we have the assertion (1).

By using the facts $s_i > \widetilde{s}_i$ and $s_i>0$,
it follows from Proposition \ref{prop-equiv-norm} and the assertion (1) that
\[
\|\widetilde{m}_j\|_{W^{(\widetilde{s}_1,\widetilde{s}_2)}_i}
\lesssim
\|\widetilde{m}_j\|_{B^{(s_1,s_2)}_i}
\lesssim
\sup_{j \in \Z}\|m_j\|_{B^{(s_1,s_2)}_i}
\lesssim
\sup_{j \in \Z}\|m_j\|_{W^{(s_1,s_2)}_i},
\]
which gives the assertion (2).
\end{proof}

In the proof of Theorem 1.1, 
we shall derive the 
$L^2 \times L^2 \to L^1$ estimate 
from the $L^2 \times L^\infty \to L^2$ estimate 
by using duality 
(see Subsection \ref{section4.2}).
To do this,
we introduce
\begin{equation}\label{dual-form}
m^{*1}(\xi_1,\xi_2)
=m(-\xi_1-\xi_2, \xi_2),
\qquad
m^{*2}(\xi_1,\xi_2)
=m(\xi_1, -\xi_1-\xi_2),
\end{equation}
and prove the following.

\begin{lem}\label{change-variable}
\begin{enumerate}
\item
For each $s_i>0$, $i=1,2$,
\[
\|m^{*1}\|_{B^{(s_1,s_2)}_1}
\lesssim
\|m\|_{B^{(s_1+s_2,s_2)}_1},
\quad
\|m^{*2}\|_{B^{(s_1,s_2)}_2}
\lesssim
\|m\|_{B^{(s_1,s_1+s_2)}_2}.
\]
\item
For each $s_i > \widetilde{s}_i \ge 0$, $i=1,2$,
\begin{equation}\label{change-2}
\|m^{*1}\|_{W^{(\widetilde{s}_1,\widetilde{s}_2)}_1}
\lesssim
\|m\|_{W^{(s_1+s_2,s_2)}_1},
\quad
\|m^{*2}\|_{W^{(\widetilde{s}_1,\widetilde{s}_2)}_2}
\lesssim
\|m\|_{W^{(s_1,s_1+s_2)}_2}.
\end{equation}
\end{enumerate}
\end{lem}
\begin{proof}
By symmetry,
it is sufficient to consider the estimate for $m^{*2}$.
Moreover, we may assume that
each function of $\{\psi_{k_i}\}_{k_i \in \N_0}$
appearing in the definition of $\Delta_{\boldsymbol{k}}$
is even, namely $\psi_{k_i}(y_i)=\psi_{k_i}(-y_i)$.
By a change of variables,
\begin{align*}
\Delta_{\boldsymbol{k}}m^{*2}(\xi_1,\xi_2)
&=\frac{1}{(2\pi)^{2n}}\int_{(\R^n)^2}
e^{i(\xi_1\cdot y_1+\xi_2\cdot y_2)}
\psi_{k_1}(y_1)\psi_{k_2}(y_2)
\widehat{m}(y_1-y_2,-y_2)\, dy_1 dy_2
\\
&=\frac{1}{(2\pi)^{2n}}\int_{(\R^n)^2}
e^{i(\xi_1\cdot y_1-(\xi_1+\xi_2)\cdot y_2)}
\psi_{k_1}(y_1-y_2)\psi_{k_2}(y_2)
\widehat{m}(y_1,y_2)\, dy_1 dy_2 ,
\end{align*}
and then
\begin{align*}
&\big\|\|\Delta_{\boldsymbol{k}}m^{*2}(\xi_1,\xi_2)
\|_{L^2_{\xi_2}}\big\|_{L^{\infty}_{\xi_1}}
\\
&\le
\bigg\|\Big\|\int_{(\R^n)^2}
e^{i(\xi_1\cdot y_1+\xi_2 \cdot y_2)}
\psi_{k_1}(y_1-y_2)\psi_{k_2}(y_2)
\widehat{m}(y_1,y_2)\, dy_1 dy_2\Big\|_{L^2_{\xi_2}}
\bigg\|_{L^{\infty}_{\xi_1}} .
\end{align*}
We decompose the function inside the
$L^{\infty}_{\xi_1}(L^2_{\xi_2})$-norm
in the right hand side of the above inequality as
\begin{align*}
&\sum_{\widetilde{k}_1 \in \N_0}
\int_{(\R^n)^2}
e^{i(\xi_1\cdot y_1+\xi_2 \cdot y_2)}
\psi_{k_1}(y_1-y_2)\psi_{\widetilde{k}_1}(y_1)\psi_{k_2}(y_2)
\widehat{m}(y_1,y_2) dy_1 dy_2
\\
&=\Big(\sum_{\widetilde{k}_1 \le k_2+2}
+\sum_{\widetilde{k}_1 > k_2+2}\Big)
\int_{(\R^n)^2}
e^{i(\xi_1\cdot y_1+\xi_2 \cdot y_2)}
\psi_{k_1}(y_1-y_2)\psi_{\widetilde{k}_1}(y_1)\psi_{k_2}(y_2)
\widehat{m}(y_1,y_2) dy_1 dy_2 .
\end{align*}
It follows from the support property of $\psi_{k_i}$ that
if $y_1 \in \mathrm{supp}\, \psi_{\widetilde{k}_1}$
and $y_2 \in \mathrm{supp}\, \psi_{k_2}$,
then $|y_1-y_2| \le 2^{k_2+4}$ for $\widetilde{k}_1 \le k_2+2$
and $2^{\widetilde{k}_1-2} \le |y_1-y_2| \le 2^{\widetilde{k}_1+2}$
for $\widetilde{k}_1 > k_2+2$.
Thus, we can restrict the above sums to
\[
\Bigg(\sum_{\substack{\widetilde{k}_1 \in \N_0 \,:\,
\widetilde{k}_1 \le k_2+2 \\ k_1-1 < k_2+4}}
+\sum_{\substack{\widetilde{k}_1\in \N_0 \,:\,
\widetilde{k}_1 > k_2+2 \\
k_1-1<\widetilde{k}_1+2, \, k_1+1>\widetilde{k}_1-2}}\Bigg)
\cdots \cdots .
\]

By the Fourier inversion formula,
\begin{align*}
&\int_{(\R^n)^2}
e^{i(\xi_1\cdot y_1+\xi_2 \cdot y_2)}
\psi_{k_1}(y_1-y_2)\psi_{\widetilde{k}_1}(y_1)\psi_{k_2}(y_2)
\widehat{m}(y_1,y_2)\, dy_1 dy_2
\\
&=\int_{(\R^n)^2}
e^{i(\xi_1\cdot y_1+\xi_2 \cdot y_2)}
\Big(\frac{1}{(2\pi)^n}\int_{\R^n}
e^{i(y_1-y_2)\cdot\eta}\widehat{\psi_{k_1}}(\eta)\, d\eta \Big)
\\
&\qquad \qquad \times
\psi_{\widetilde{k}_1}(y_1)\psi_{k_2}(y_2)
\widehat{m}(y_1,y_2)\, dy_1 dy_2
\\
&=(2\pi)^n\int_{\R^n}2^{k_1 n}
\widehat{\psi}(2^{k_1}\eta)
[\psi_{\widetilde{k}_1}(D_{\xi_1})\psi_{k_2}(D_{\xi_2})
m](\xi_1+\eta,\xi_2-\eta)\, d\eta ,
\end{align*}
where $\widehat{\psi}$ is replaced by $\widehat{\psi_0}$
if $k_1=0$.
Hence,
$2^{\boldsymbol{k} \cdot \boldsymbol {s}}
\big\|\|\Delta_{\boldsymbol{k}}m^{*2}(\xi_1,\xi_2)
\|_{L^2_{\xi_2}}\big\|_{L^{\infty}_{\xi_1}}$ is estimated by
\begin{align*}
&\Bigg(\sum_{\substack{\widetilde{k}_1 \,:\,
\widetilde{k}_1 \le k_2+2 \\ k_1< k_2+5}}
+\sum_{\substack{\widetilde{k}_1 \,:\,
\widetilde{k}_1 > k_2+2 \\
|\widetilde{k}_1-k_1| < 3}}\Bigg)
2^{k_1s_1+k_2s_2}
\left\|\left\|[\psi_{\widetilde{k}_1}(D_{\xi_1})\psi_{k_2}(D_{\xi_2})
m](\xi_1,\xi_2)\right\|_{L^2_{\xi_2}}\right\|_{L^{\infty}_{\xi_1}}
\\
&\lesssim
\sum_{\widetilde{k}_1 \,:\,
\widetilde{k}_1 \le k_2+2}
2^{k_2(s_1+s_2)}\left\|\left\|
\Delta_{(\widetilde{k}_1,k_2)}m
\right\|_{L^2_{\xi_2}}\right\|_{L^{\infty}_{\xi_1}}
\\
&\qquad \qquad
+\sum_{\widetilde{k}_1 \,:\, |\widetilde{k}_1-k_1| < 3}
2^{\widetilde{k}_1s_1+k_2s_2}\left\|\left\|
\Delta_{(\widetilde{k}_1,k_2)}m
\right\|_{L^2_{\xi_2}}\right\|_{L^{\infty}_{\xi_1}}
\\
&=\sum_{\widetilde{k}_1 \,:\,
\widetilde{k}_1 \le k_2+2}
2^{-\widetilde{k}_1s_1}
\Big(2^{\widetilde{k}_1s_1+k_2(s_1+s_2)}\left\|\left\|
\Delta_{(\widetilde{k}_1,k_2)}m
\right\|_{L^2_{\xi_2}}\right\|_{L^{\infty}_{\xi_1}}\Big)
\\
&\qquad \qquad
+\sum_{\widetilde{k}_1 \,:\, |\widetilde{k}_1-k_1| < 3}
2^{\widetilde{k}_1s_1+k_2s_2}\left\|\left\|
\Delta_{(\widetilde{k}_1,k_2)}m
\right\|_{L^2_{\xi_2}}\right\|_{L^{\infty}_{\xi_1}}
\\
&\lesssim
\sup_{\widetilde{k}_1, k_2 \in \N_0}
2^{\widetilde{k}_1s_1+k_2(s_1+s_2)}\left\|\left\|
\Delta_{(\widetilde{k}_1,k_2)}m
\right\|_{L^2_{\xi_2}}\right\|_{L^{\infty}_{\xi_1}}
=\|m\|_{B^{(s_1,s_1+s_2)}_2},
\end{align*}
where we used the assumption $s_1>0$
in the second inequality.
Therefore, we have the assertion (1).
As a result of it,
the assertion (2) can be proved in the same way
as in the last part of the proof of Lemma \ref{remove-cut}.
\end{proof}

We end this section by giving the following remark
mentioned before Theorem \ref{mainthm}.

\begin{rem}\label{bounded-continuous}
Let $s_1>0$ and $s_2>n/2$.
Assume that $m \in \calS'(\R^n \times \R^n)$
satisfies $\|m\|_{W^{(s_1,s_2)}_2}<\infty$.
Using $\Delta_{\boldsymbol{k}}$,
$\boldsymbol{k} \in \N_0^2$, given in the beginning of this section,
we decompose 
$m=\sum\Delta_{\boldsymbol{k}}m$. Since
$\mathrm{supp}\,
\big(\Delta_{\boldsymbol{k}}m(\xi_1,\cdot)\big)^{\wedge}
\subset \{|y_2| \le 2^{k_2+1}\}$
for each $\xi_1 \in \R^n$,
we have
\[
\|\Delta_{\boldsymbol{k}}m(\xi_1,\xi_2)\|_{L^{\infty}_{\xi_2}}
\lesssim 2^{k_2 n/2}
\|\Delta_{\boldsymbol{k}}m(\xi_1,\xi_2)\|_{L^{2}_{\xi_2}}
\]
(see \cite[Remark 1.3.2/1]{Triebel}).
Then,
\begin{align*}
\|m\|_{L^{\infty}}
&\le \sum_{\boldsymbol{k} \in \N_0^2}
\|\Delta_{\boldsymbol{k}}m\|_{L^{\infty}}
\lesssim
\sum_{\boldsymbol{k} \in \N_0^2}
2^{k_2 n/2}\Big\|\|\Delta_{\boldsymbol{k}}m
(\xi_1,\xi_2)\|_{L^2_{\xi_2}}\Big\|_{L^{\infty}_{\xi_1}}
\\
&=\sum_{\boldsymbol{k} \in \N_0^2}
2^{-k_1s_1-k_2(s_2-n/2)}
\Big(2^{k_1s_1+k_2s_2}
\Big\|\Delta_{\boldsymbol{k}}m
(\xi_1,\xi_2)\|_{L^2_{\xi_2}}\Big\|_{L^{\infty}_{\xi_1}}\Big)
\lesssim \|m\|_{B^{(s_1,s_2)}_2}.
\end{align*}
Hence, it follows from Proposition \ref{equiv-norm} that
$\|m\|_{L^{\infty}} \lesssim \|m\|_{W^{(s_1,s_2)}_2}$.
Moreover, the above inequality says that
$\sum \|\Delta_{\boldsymbol{k}}m\|_{L^{\infty}}<\infty$.
Combining this with the fact that
each $\Delta_{\boldsymbol{k}}m(\xi_1,\xi_2)$ is continuous,
we see that
$m(\xi_1,\xi_2)$ coincides with a continuous function
on $\R^n \times \R^n$ almost everywhere.
\end{rem}

\section{Proof of Theorem \ref{mainthm}}
\label{section4}

In this section, we shall prove Theorem \ref{mainthm}.
We first consider the $L^2 \times L^{\infty} \to L^2$ boundedness,
and next prove the $L^2 \times L^2 \to L^1$ one by duality.

\subsection{The boundedness from $L^2 \times L^{\infty}$ to $L^2$}
\label{section4.1}
The goal of this subsection is to show that
if $s_1>0$ and $s_2>n/2$,
then
\begin{equation}\label{L^2L^{infty}-goal}
\|T_m\|_{L^2 \times L^\infty \to L^2}
\lesssim
\sup_{j \in \Z}\|m_j\|_{W^{(s_1,s_2)}_2},
\end{equation}
where $m_j$ is defined by \eqref{def-mj}.
By duality,
this follows from the estimate
\begin{equation}\label{L^2L^{infty}-dual}
\left|\int_{\R^n}T_m(f_1,f_2)(x)g(x)\, dx\right|
\lesssim\Big( \sup_{j \in \Z}\|m_j\|_{W^{(s_1,s_2)}_2}\Big)
\|f_1\|_{L^2}\|f_2\|_{L^{\infty}}\|g\|_{L^2}.
\end{equation}

We first observe that it is sufficient to consider the case
where $\mathrm{supp}\, m$ is included in a cone.
If $(\xi_1,\xi_2)$ belongs to the unit sphere
$\Sigma$ of $\R^n \times \R^n$,
then at least two of the three vectors $\xi_1$,
$\xi_2$ and $\xi_1+\xi_2$
are not equal to $0$ as elements of $\R^n$.
By the compactness of $\Sigma$,
this implies that there exists a constant $c>0$ such that $\Sigma$ is
covered by the three open sets
\begin{align*}
&V_0=\{(\xi_1,\xi_2) \in \Sigma \,:\, |\xi_1|>c, \ |\xi_2|>c\},
\\
&V_1=\{(\xi_1,\xi_2) \in \Sigma \,:\, |\xi_1|>c, \ |\xi_1+\xi_2|>c\},
\\
&V_2=\{(\xi_1,\xi_2) \in \Sigma \,:\, |\xi_2|>c, \ |\xi_1+\xi_2|>c\}.
\end{align*}
We write
\[
\Gamma(V_i)=\{(\xi_1,\xi_2) \in \R^n \times \R^n \setminus \{(0,0)\}
\,:\, (\xi_1,\xi_2)/|(\xi_1,\xi_2)| \in V_i\},
\quad i=0,1,2,
\]
and take functions $\Phi_i$ on $\R^n \times \R^n \setminus \{(0,0)\}$
such that each $\Phi_i$ is homogeneous of degree $0$,
smooth away from the origin,
$\mathrm{supp}\, \Phi_i \subset \Gamma(V_i)$,
and $\sum_{i=0}^2 \Phi_i(\xi_1,\xi_2)=1$, $(\xi_1,\xi_2) \neq (0,0)$.
Then, the multiplier $m$ can be written as
\[
m(\xi_1,\xi_2)
=\sum_{i=0}^2
m(\xi_1,\xi_2)\Phi_{i}(\xi_1,\xi_2),
\]
and it is sufficient to prove the boundedness
of each $T_{m\Phi_i}$.
In fact, if we can prove
\eqref{L^2L^{infty}-dual} with $s_i$ and $m$
replaced by $s_i-\epsilon$ and $m\Phi_i$,
where $\epsilon>0$ is a sufficiently small number
such that $s_1-\epsilon>0$ and $s_2-\epsilon>n/2$,
then it follows from \eqref{remove-cut-2} that
\begin{align*}
\|T_{m}\|_{L^2 \times L^{\infty} \to L^2}
&\le \sum_{i=0}^{2}\|T_{m\Phi_i}\|_{L^2 \times L^{\infty} \to L^2}
\\
&\lesssim
\sum_{i=0}^{2}
\sup_{j \in \Z}\|(m\Phi_i)_j\|_{W^{(s_1-\epsilon,s_2-\epsilon)}_2}
\lesssim
\sup_{j \in \Z}\|m_j\|_{W^{(s_1,s_2)}_2},
\end{align*}
which is \eqref{L^2L^{infty}-goal}.
Hence, writing simply $m$ instead of $m\Phi_i$,
we may assume that $\mathrm{supp}\, m$
is included in one of $\Gamma(V_i)$.

In any case,
we decompose $m$ as
\[
m(\xi_1,\xi_2)=\sum_{j \in \Z}
m(\xi_1,\xi_2)\Psi(2^{-j}\xi_1, 2^{-j}\xi_2)
=\sum_{j \in \Z}m_{(j)}(\xi_1,\xi_2),
\]
where $\Psi \in \calS(\R^{2n})$
is as in \eqref{LP-homogeneous} with $d=2n$
and $m_{(j)}(\xi_1,\xi_2)=m_j(2^{-j}\xi_1,2^{-j}\xi_2)$.
Thus,
\[
\int_{\R^n}T_m(f_1,f_2)(x)g(x)\, dx
=\sum_{j \in \Z}
\int_{\R^n}T_{m_{(j)}}(f_1,f_2)(x)g(x)\, dx.
\]
Hereafter, let us consider the three cases separately.
The cases $\mathrm{supp}\,m \subset \Gamma (V_{i})$, $i=0,1$,
can be handled in the same way as in \cite{Miyachi-Tomita-TMJ},
and we need a new idea only for the case
$\mathrm{supp}\,m \subset \Gamma (V_{2})$.
We use the following notations:
$\calA_0$ denotes the set of even functions $\varphi \in \calS(\R^n)$
for which $\mathrm{supp}\, \varphi$ is compact
and $\varphi=1$ on some neighborhood of the origin;
$\calA_1$ denotes the set of even functions $\psi \in \calS(\R^n)$
for which $\mathrm{supp}\, \psi$ is a compact subset of $\R^n \setminus \{0\}$.

\medskip
{\it The case $\mathrm{supp}\,m \subset \Gamma (V_{0})$.}
In this case, if $(\xi_1,\xi_2) \in \mathrm{supp}\, m_{(j)}$,
then $|\xi_1| \approx |\xi_2| \approx |(\xi_1,\xi_2)| \approx 2^j$.
Hence, using appropriate functions
$\varphi \in \calA_0$ and $\psi \in \calA_1$,
we can write
\[
m_{(j)}(\xi_1,\xi_2)
=m_{(j)}(\xi_1,\xi_2)
\varphi(2^{-j}(\xi_1+\xi_2))\psi(2^{-j}\xi_1)\psi(2^{-j}\xi_2) .
\]
This gives
\[
\int_{\R^n}T_{m_{(j)}}(f_1,f_2)(x)g(x)\, dx
=\int_{\R^n}T_{{m_{(j)}}}
(\psi(2^{-j}D)f_1, \psi(2^{-j}D)f_2)(x)
\varphi(2^{-j}D)g(x)\, dx.
\]
Since $m_{(j)}(\xi_1,\xi_2)=m_j(2^{-j}\xi_1,2^{-j}\xi_2)$,
it follows from Lemma \ref{TMJ-pointwise} that
\begin{align*}
&|T_{m_{(j)}}
(\psi(2^{-j}D)f_1, \psi(2^{-j}D)f_2)(x)|
\\
&\lesssim
\left\|\int_{\R^n}e^{ix\cdot\xi_1}
\langle D_{\xi_2}\rangle^{s_2}
m_j(2^{-j}\xi_1,\xi_2)\calF[\psi(2^{-j}D)f_1](\xi_1)\, d\xi_1
\right\|_{L^2_{\xi_2}}
\\
&\qquad \qquad \times
(\zeta_j*|\psi(2^{-j}D)f_2|^2)(x)^{1/2} ,
\end{align*}
where $\zeta_j(x)=2^{jn}(1+2^j|x|)^{-2s_2}$.
Then, by Schwarz's inequality,
\begin{equation}\label{V0-case-1}
\begin{split}
&\left|\int_{\R^n}T_{m}(f_1,f_2)(x)g(x)\, dx\right|
\\
&\lesssim
\bigg(\sum_{j \in \Z}
\bigg\|\Big\|\int_{\R^n}e^{ix\cdot\xi_1}
\langle D_{\xi_2}\rangle^{s_2}
m_j(2^{-j}\xi_1,\xi_2)\calF[\psi(2^{-j}D)f_1](\xi_1)\, d\xi_1
\Big\|_{L^2_{\xi_2}}\bigg\|_{L^2_x}^2\bigg)^{1/2}
\\
&\qquad \qquad \times
\Big(\sum_{j \in \Z}
\Big\|\varphi(2^{-j}D)g(x)
(\zeta_j*|\psi(2^{-j}D)f_2|^2)(x)^{1/2}\Big\|_{L^2_x}^2\Big)^{1/2}.
\end{split}
\end{equation}

Changing the order of integrals,
and using Plancherel's theorem,
we have
\begin{align*}
&\bigg\|\Big\|\int_{\R^n}e^{ix\cdot\xi_1}
\langle D_{\xi_2}\rangle^{s_2}
m_j(2^{-j}\xi_1,\xi_2)\calF[\psi(2^{-j}D)f_1](\xi_1)\, d\xi_1
\Big\|_{L^2_{\xi_2}}\bigg\|_{L^2_x}
\\
&=(2\pi)^{n/2}\bigg\|\Big\|
\langle D_{\xi_2}\rangle^{s_2}
m_j(2^{-j}\xi_1,\xi_2)\calF[\psi(2^{-j}D)f_1](\xi_1)
\Big\|_{L^2_{\xi_1}}\bigg\|_{L^2_{\xi_2}}
\\
&\lesssim
\bigg\|\Big\|
\langle D_{\xi_2}\rangle^{s_2}
m_j(\xi_1,\xi_2)
\Big\|_{L^2_{\xi_2}}\bigg\|_{L^{\infty}_{\xi_1}}
\|\calF[\psi(2^{-j}D)f_1](\xi_1)\|_{L^2_{\xi_1}}
\\
&\approx
\|m_j\|_{W^{(0,s_2)}_2}\|\psi(2^{-j}D)f_1\|_{L^2}
\lesssim
\|m_j\|_{W^{(s_1,s_2)}_2}\|\psi(2^{-j}D)f_1\|_{L^2},
\end{align*}
where we used \eqref{rem-equiv-norm-1}
in the last inequality.
Thus, since $\psi$ is a Schwartz function
whose support is a compact subset of $\R^n \setminus \{0\}$,
the quantity concerning $f_1$
in the right hand side of \eqref{V0-case-1}
is estimated by
\[
\left(\sup_{j \in \Z}
\|m_j\|_{W^{(s_1,s_2)}_2}\right)
\left(\sum_{j \in \Z}\|\psi(2^{-j}D)f_1\|_{L^2}^2\right)^{1/2}
\lesssim
\left(\sup_{j \in \Z}
\|m_j\|_{W^{(s_1,s_2)}_2}\right)
\|f_1\|_{L^2}.
\]
On the other hand,
by using the inequality $|\varphi(2^{-j}D)g| \lesssim \zeta_j*|g|$,
it follows from Lemma \ref{CJM-Carleson} that
the quantity concerning $f_2$ and $g$ in the right hand side of
\eqref{V0-case-1} is estimated by
$\|f_2\|_{BMO}\|g\|_{L^2} \lesssim \|f_2\|_{L^{\infty}}\|g\|_{L^2}$.
Therefore, \eqref{L^2L^{infty}-dual} is obtained.

\medskip
{\it The case $\mathrm{supp}\,m \subset \Gamma (V_{1})$.}
In this case, if $(\xi_1,\xi_2) \in \mathrm{supp}\, m_{(j)}$,
then $|\xi_1| \approx |\xi_1+\xi_2| \approx |(\xi_1,\xi_2)| \approx 2^j$.
Hence, using an appropriate function $\psi \in \calA_1$,
we can write
\[
m_{(j)}(\xi_1,\xi_2)
=m_{(j)}(\xi_1,\xi_2)
\psi(2^{-j}(\xi_1+\xi_2))\psi(2^{-j}\xi_1) .
\]
This gives
\[
\int_{\R^n}T_{m_{(j)}}(f_1,f_2)(x)g(x)\, dx
=\int_{\R^n}T_{{m_{(j)}}}
(\psi(2^{-j}D)f_1, f_2)(x)
\psi(2^{-j}D)g(x)\, dx.
\]
It follows from Lemma \ref{TMJ-pointwise}
and the inequality
$(\zeta_j*|f_2|^2)^{1/2} \lesssim \|f_2\|_{L^{\infty}}$ that
\begin{align*}
&|T_{m_{(j)}}
(\psi(2^{-j}D)f_1, f_2)(x)|
\\
&\lesssim
\left\|\int_{\R^n}e^{ix\cdot\xi_1}
\langle D_{\xi_2}\rangle^{s_2}
m_j(2^{-j}\xi_1,\xi_2)\calF[\psi(2^{-j}D)f_1](\xi_1)\, d\xi_1
\right\|_{L^2_{\xi_2}}
\|f_2\|_{L^{\infty}} .
\end{align*}
Then,
instead of \eqref{V0-case-1},
we have
\begin{align*}
&\left|\int_{\R^n}T_{m}(f_1,f_2)(x)g(x)\, dx\right|
\\
&\lesssim
\bigg(\sum_{j \in \Z}
\bigg\|\Big\|\int_{\R^n}e^{ix\cdot\xi_1}
\langle D_{\xi_2}\rangle^{s_2}
m_j(2^{-j}\xi_1,\xi_2)\calF[\psi(2^{-j}D)f_1](\xi_1)\, d\xi_1
\Big\|_{L^2_{\xi_2}}\bigg\|_{L^2_x}^2\bigg)^{1/2}
\\
&\qquad \qquad \times
\|f_2\|_{L^{\infty}}
\Big(\sum_{j \in \Z}
\Big\|\psi(2^{-j}D)g\Big\|_{L^2_x}^2\Big)^{1/2}.
\end{align*}
Since the quantity concerning $f_1$ in the right hand side of
the last inequality is the same as before,
and the quantity concerning $g$ is estimated by $\|g\|_{L^2}$,
we have \eqref{L^2L^{infty}-dual}.

\medskip
{\it The case $\mathrm{supp}\,m \subset \Gamma (V_{2})$.}
In this case, if $(\xi_1,\xi_2) \in \mathrm{supp}\, m_{(j)}$,
then $|\xi_2| \approx |\xi_1+\xi_2| \approx |(\xi_1,\xi_2)| \approx 2^j$.
Hence, using appropriate functions $\varphi \in \calA_0$
and $\psi \in \calA_1$,
we can write
\[
m_{(j)}(\xi_1,\xi_2)
=m_{(j)}(\xi_1,\xi_2)
\psi(2^{-j}(\xi_1+\xi_2))\varphi(2^{-j}\xi_1)\psi(2^{-j}\xi_2)
\]
This gives
\[
\int_{\R^n}T_{m_{(j)}}(f_1,f_2)(x)g(x)\, dx
=\int_{\R^n}T_{{m_{(j)}}}
(\varphi(2^{-j}D)f_1, \psi(2^{-j}D)f_2)(x)
\psi(2^{-j}D)g(x)\, dx.
\]
By Lemma \ref{TMJ-pointwise},
\begin{align*}
&|T_{m_{(j)}}
(\varphi(2^{-j}D)f_1, \psi(2^{-j}D)f_2)(x)|
\\
&\lesssim
\left\|\int_{\R^n}e^{ix\cdot\xi_1}
\langle D_{\xi_2}\rangle^{s_2}
m_j(2^{-j}\xi_1,\xi_2)\calF[\varphi(2^{-j}D)f_1](\xi_1)\, d\xi_1
\right\|_{L^2_{\xi_2}}
\\
&\qquad \qquad \times
(\zeta_j*|\psi(2^{-j}D)f_2|^2)(x)^{1/2} .
\end{align*}
We divide the $L^2_{\xi_2}$-norm concerning $f_1$
in the right hand side of the last inequality into
\begin{equation}\label{V2-case-1}
\begin{split}
&\left\|\int_{\R^n}e^{ix\cdot\xi_1}
\langle D_{\xi_2}\rangle^{s_2}
m_j(0,\xi_2)\calF[\varphi(2^{-j}D)f_1](\xi_1)\, d\xi_1
\right\|_{L^2_{\xi_2}}
\\
&\quad
+\left\|\int_{\R^n}e^{ix\cdot\xi_1}
\langle D_{\xi_2}\rangle^{s_2}
(m_j(2^{-j}\xi_1,\xi_2)-m_j(0,\xi_2))\calF[\varphi(2^{-j}D)f_1](\xi_1)\, d\xi_1
\right\|_{L^2_{\xi_2}}
\\
&=(2\pi)^n
\|\langle D_{\xi_2}\rangle^{s_2}
m_j(0,\xi_2)\|_{L^2_{\xi_2}}
|\varphi(2^{-j}D)f_1(x)|
\\
&\quad
+\left\|\int_{\R^n}e^{ix\cdot\xi_1}
\langle D_{\xi_2}\rangle^{s_2}
\tau_j(\xi_1,\xi_2)\calF[\varphi(2^{-j}D)f_1](\xi_1)\, d\xi_1
\right\|_{L^2_{\xi_2}},
\end{split}
\end{equation}
where
\[
\tau_j(\xi_1,\xi_2)
=m_j(2^{-j}\xi_1,\xi_2)-m_j(0,\xi_2).
\]
In the rest of the proof,
we assume that
$\sup_{j \in \Z}\|m_j\|_{W^{(s_1,s_2)}_2}=1$
for the sake of simplicity.

For the former term
in the right hand side of \eqref{V2-case-1},
the quantity we have to consider is
\begin{align*}
&\sum_{j \in \Z}\|\langle D_{\xi_2}\rangle^{s_2}
m_j(0,\xi_2)\|_{L^2_{\xi_2}}
\\
&\qquad \times
\int_{\R^n}
|\varphi(2^{-j}D)f_1(x)|
(\zeta_j*|\psi(2^{-j}D)f_2|^2)(x)^{1/2}
|\psi(2^{-j}D)g(x)|\, dx,
\end{align*}
but we can easily handle this.
In fact, since
\[
\|\langle D_{\xi_2}\rangle^{s_2}
m_j(0,\xi_2)\|_{L^2_{\xi_2}}
\le \left\|\|\langle D_{\xi_2}\rangle^{s_2}
m_j(\xi_1,\xi_2)\|_{L^2_{\xi_2}}\right\|_{L^{\infty}_{\xi_1}}
\lesssim
\|m_j\|_{W^{(s_1,s_2)}_2}
\le 1,
\]
where we used  \eqref{rem-equiv-norm-1}
for the second inequality,
it follows from Schwarz's inequality that
the above quantity is estimated by
\[
\left(\sum_{j \in \Z}\|\varphi(2^{-j}D)f_1 \,
(\zeta_j*|\psi(2^{-j}D)f_2|^2)^{1/2}\|_{L^2}^2\right)^{1/2}
\left(\sum_{j \in \Z}
\|\psi(2^{-j}D)g\|_{L^2}^2\right)^{1/2},
\]
and this is estimated by
the right hand side of \eqref{L^2L^{infty}-dual}
from Lemma \ref{CJM-Carleson}.

For the latter term
in the right hand side of \eqref{V2-case-1},
the quantity we have to consider is
\begin{align*}
&\sum_{j \in \Z}\int_{\R^n}\left\|\int_{\R^n}e^{ix\cdot\xi_1}
\langle D_{\xi_2}\rangle^{s_2}
\tau_j(\xi_1,\xi_2)\calF[\varphi(2^{-j}D)f_1](\xi_1)\, d\xi_1
\right\|_{L^2_{\xi_2}}
\\
&\qquad \qquad \times
(\zeta_j*|\psi(2^{-j}D)f_2|^2)(x)^{1/2}
|\psi(2^{-j}D)g(x)|\, dx.
\end{align*}
By Schwarz's inequality, this is estimated by
\begin{align*}
&\bigg(\sum_{j \in \Z}\bigg\|\Big\|\int_{\R^n}e^{ix\cdot\xi_1}
\langle D_{\xi_2}\rangle^{s_2}
\tau_j(\xi_1,\xi_2)\calF[\varphi(2^{-j}D)f_1](\xi_1)\, d\xi_1
\Big\|_{L^2_{\xi_2}}\bigg\|_{L^2_x}^2\bigg)^{1/2}
\\
&\qquad \qquad \times
\left(\sup_{j \in \Z}\|(\zeta_j*|\psi(2^{-j}D)f_2|^2)^{1/2}\|_{L^{\infty}}
\right)
\left(\sum_{j \in \Z}\|\psi(2^{-j}D)g\|_{L^2}^2\right)^{1/2}
\\
&\lesssim
\bigg(\sum_{j \in \Z}\bigg\|\Big\|\int_{\R^n}e^{ix\cdot\xi_1}
\langle D_{\xi_2}\rangle^{s_2}
\tau_j(\xi_1,\xi_2)\calF[\varphi(2^{-j}D)f_1](\xi_1)\, d\xi_1
\Big\|_{L^2_{\xi_2}}\bigg\|_{L^2_x}^2\bigg)^{1/2}
\\
&\qquad \qquad \times
\|f_2\|_{L^{\infty}}\|g\|_{L^2}.
\end{align*}
Then, the remaining task is to show that
the term concerning $f_1$ in the last line is estimated by $\|f_1\|_{L^2}$,
and this is done by proving
\begin{equation}\label{V2-case-2}
\sup_{\xi_1 \in \R^n}\bigg(\sum_{j \in \Z}
\Big\|\langle D_{\xi_2}\rangle^{s_2}
\tau_j(\xi_1,\xi_2)\varphi(2^{-j}\xi_1)
\Big\|_{L^2_{\xi_2}}^2\bigg)^{1/2}
\lesssim 1.
\end{equation}
Indeed, once \eqref{V2-case-2} is obtained,
by changing the order of integrals,
and applying Plancherel's theorem to the $L^2_x$-norm,
we have
\begin{align*}
&\bigg(\sum_{j \in \Z}\bigg\|\Big\|\int_{\R^n}e^{ix\cdot\xi_1}
\langle D_{\xi_2}\rangle^{s_2}
\tau_j(\xi_1,\xi_2)\calF[\varphi(2^{-j}D)f_1](\xi_1)\, d\xi_1
\Big\|_{L^2_{\xi_2}}\bigg\|_{L^2_x}^2\bigg)^{1/2}
\\
&=(2\pi)^{n/2}
\bigg(\sum_{j \in \Z}\bigg\|\Big\|
\langle D_{\xi_2}\rangle^{s_2}
\tau_j(\xi_1,\xi_2)\calF[\varphi(2^{-j}D)f_1](\xi_1)
\Big\|_{L^2_{\xi_1}}\bigg\|_{L^2_{\xi_2}}^2\bigg)^{1/2}
\\
&=(2\pi)^{n/2}
\bigg\|\bigg(\sum_{j \in \Z}
\Big\|\langle D_{\xi_2}\rangle^{s_2}
\tau_j(\xi_1,\xi_2)\varphi(2^{-j}\xi_1)
\Big\|_{L^2_{\xi_2}}^2\bigg)^{1/2} \,
\widehat{f_1}(\xi_1)\bigg\|_{L^2_{\xi_1}}
\lesssim
\|f_1\|_{L^2}.
\end{align*}

We shall prove \eqref{V2-case-2},
and this follows from the estimate
\begin{equation}\label{V2-case-3}
\left\|\langle D_{\xi_2}\rangle^{s_2}
\tau_j(\xi_1,\xi_2)\varphi(2^{-j}\xi_1)
\right\|_{L^2_{\xi_2}}
\lesssim
\min\left\{|2^{-j}\xi_1|^{-1}, \,
|2^{-j}\xi_1|^{\epsilon} \right\},
\end{equation}
where $0<\epsilon<s_1$.
It is easy to estimate the left hand side
of \eqref{V2-case-3} from above by $|2^{-j}\xi|^{-1}$.
In fact,
since $\varphi$ is rapidly decreasing,
it is estimated by
\begin{align*}
\left\|\langle D_{\xi_2}\rangle^{s_2}
\tau_j(\xi_1,\xi_2)\right\|_{L^2_{\xi_2}}
|2^{-j}\xi_1|^{-1}
&\le
2\left\|\left\|\langle D_{\xi_2}\rangle^{s_2}
m_j(\xi_1,\xi_2)\right\|_{L^2_{\xi_2}}\right\|_{L^{\infty}_{\xi_1}}
|2^{-j}\xi_1|^{-1}
\\
&\lesssim
\left(\sup_{j \in \Z}
\|m_j\|_{W^{(s_1,s_2)}_2}\right)
|2^{-j}\xi_1|^{-1}
= |2^{-j}\xi_1|^{-1} .
\end{align*}
In order to estimate it by $|2^{-j}\xi_1|^{\epsilon}$,
we decompose $\tau_j$ as
\[
\tau_j(\xi_1,\xi_2)
=\sum_{k \in \N_0}
\sigma_{j,k}(\xi_1,\xi_2)
\]
with
\[
\sigma_{j,k}(\xi_1,\xi_2)
=[\psi_{k}(D_{\xi_1})m_j](2^{-j}\xi_1,\xi_2)
-[\psi_{k}(D_{\xi_1})m_j](0,\xi_2),
\]
where $\{\psi_k\}_{k \in \N_0}$
is the same as in the beginning of Section \ref{section3},
and prove
\begin{equation}\label{V2-case-4}
\begin{split}
&\|\langle D_{\xi_2}\rangle^{s_2}
\sigma_{j,k}(\xi_1,\xi_2)\|_{L^2_{\xi_2}}
\\
&\lesssim
\min\{1, \, |2^{-j}\xi_1|2^k\}
\left\|\left\|\psi_{k}(D_{\xi_1})
\langle D_{\xi_2}\rangle^{s_2}
m_j(\xi_1,\xi_2)\right\|_{L^2_{\xi_2}}
\right\|_{L^{\infty}_{\xi_1}} .
\end{split}
\end{equation}
Once this is obtained,
the left hand side of \eqref{V2-case-3}
is estimated by
\begin{align*}
&\left\|\langle D_{\xi_2}\rangle^{s_2}
\tau_j(\xi_1,\xi_2)
\right\|_{L^2_{\xi_2}}
\le
\sum_{k \in \N_0}
\left\|\langle D_{\xi_2}\rangle^{s_2}
\sigma_{j,k}(\xi_1,\xi_2)
\right\|_{L^2_{\xi_2}}^{1-\epsilon}
\left\|\langle D_{\xi_2}\rangle^{s_2}
\sigma_{j,k}(\xi_1,\xi_2)
\right\|_{L^2_{\xi_2}}^{\epsilon}
\\
&\lesssim
\sum_{k \in \N_0}
\left\|\left\|\psi_{k}(D_{\xi_1})
\langle D_{\xi_2}\rangle^{s_2}
m_j(\xi_1,\xi_2)\right\|_{L^2_{\xi_2}}
\right\|_{L^{\infty}_{\xi_1}}^{1-\epsilon}
\\
&\qquad \qquad \times
\left( |2^{-j}\xi_1|2^k
\left\|\left\|\psi_{k}(D_{\xi_1})
\langle D_{\xi_2}\rangle^{s_2}
m_j(\xi_1,\xi_2)\right\|_{L^2_{\xi_2}}
\right\|_{L^{\infty}_{\xi_1}}\right)^{\epsilon}
\\
&=|2^{-j}\xi_1|^{\epsilon}
\sum_{k \in \N_0}
2^{k(\epsilon-s_1)}
\left(2^{ks_1}\left\|\left\|\psi_{k}(D_{\xi_1})
\langle D_{\xi_2}\rangle^{s_2}
m_j(\xi_1,\xi_2)\right\|_{L^2_{\xi_2}}
\right\|_{L^{\infty}_{\xi_1}}\right)
\\
&\lesssim
|2^{-j}\xi_1|^{\epsilon}
\left(\sup_{k \in \N_0}\|m_j\|_{W^{(s_1,s_2)}_2}\right)
=
|2^{-j}\xi_1|^{\epsilon},
\end{align*}
where we used the fact $\epsilon<s_1$
and \eqref{rem-equiv-norm-2} in the last inequality.

We finally prove \eqref{V2-case-4}.
The estimate with the factor $1$ is just obtained by the triangle inequality,
that is, the left hand side of \eqref{V2-case-4}
is estimated by
\begin{align*}
&\left\|[\psi_k(D_{\xi_1})
\langle D_{\xi_2}\rangle^{s_2}
m_j](2^{-j}\xi_1,\xi_2)
\right\|_{L^2_{\xi_2}}
+\left\|[\psi_k(D_{\xi_1})
\langle D_{\xi_2}\rangle^{s_2}
m_j](0,\xi_2)
\right\|_{L^2_{\xi_2}}
\\
&\le 2\left\|\left\|\psi_k(D_{\xi_1})
\langle D\rangle^{s_2}
m_j(\xi_1,\xi_2)
\right\|_{L^2_{\xi_2}}\right\|_{L^{\infty}_{\xi_1}}.
\end{align*}
To show the estimate with the factor $|2^{-j}\xi_1| 2^k$,
we use Taylor's formula and write
\begin{align*}
&\sigma_{j,k}(\xi_1,\xi_2)
=\int_{0}^1
(2^{-j}\xi_1)\cdot
[\nabla_{\xi_1}\psi_k(D_{\xi_1})m_j](t2^{-j}\xi_1,\xi_2)\,
dt
\\
&=\int_{0}^1
(2^{-j}\xi_1)\cdot
[\nabla_{\xi_1}\widetilde{\psi}_k(D_{\xi_1})
\psi_k(D_{\xi_1})m_j](t2^{-j}\xi_1,\xi_2)\,
dt
\\
&=\int_{0}^1
(2^{-j}\xi_1)\cdot
\left(\int_{\R^n}
2^k2^{kn}
[\nabla \calF^{-1}\widetilde{\psi}](2^k\eta_1)
[\psi_k(D_{\xi_1})m_j](t2^{-j}\xi_1-\eta_1,\xi_2)\,
d\eta_1\right) dt,
\end{align*}
where  $\{\widetilde{\psi}_k\}_{k \in \N_0}$
is the same as in the proof of Proposition
\ref{prop-equiv-norm},
and $\widetilde{\psi}$ is replaced by $\widetilde{\psi}_0$
if $k=0$.
Then, by Minkowski's inequality for integrals,
the left hand side of \eqref{V2-case-4}
can be estimated by
\[
\|\nabla \calF^{-1} \widetilde{\psi}\|_{L^1}
|2^{-j}\xi_1|2^k
\left\|\left\|\psi_{k}(D_{\xi_1})
\langle D_{\xi_2}\rangle^{s_2}
m_j(\xi_1,\xi_2)\right\|_{L^2_{\xi_2}}
\right\|_{L^{\infty}_{\xi_1}},
\]
which is the desired result.
The proof is complete.

\subsection{The boundedness from $L^2 \times L^2$ to $L^1$}
\label{section4.2}
The goal of this subsection is to show that
if $s_1>0$ and $s_2>n/2$,
then
\begin{equation}\label{L^2L^2-goal}
\|T_m\|_{L^2 \times L^2 \to L^1}
\lesssim
\sup_{j \in \Z}\|m_j\|_{W^{(s_1,s_2)}_2},
\end{equation}
where $m_j$ is defined by \eqref{def-mj}.

Let $s_i', s_i''$, $i=1,2$, be such that
$0<s_1'<s_1''<s_1$, $n/2<s_2'<s_2''$,
and $s_1''+s_2''<s_2$.
By using the formula
\[
\int
T_m(f_1,f_2)(x)g(x)\, dx
=\int
T_{m^{*1}}(g,f_2)(x)f_1(x)\, dx
=\int
T_{m^{*2}}(f_1,g)(x)f_2(x)\, dx,
\]
where $m^{*i}$, $i=1,2$, are defined by \eqref{dual-form},
it follows from  duality and \eqref{L^2L^{infty}-goal} that
\[
\|T_m\|_{L^2 \times L^2 \to L^1}
=
\|T_{m^{*2}}\|_{L^2 \times L^\infty \to L^2}
\lesssim
\sup_{j \in \Z}\|(m^{*2})_j\|_{W^{(s_1',s_2')}_2}.
\]
By \eqref{change-2}
and \eqref{remove-cut-2}
with $\Phi(\xi_1,\xi_2)=1$
and $\widetilde{\Psi}(\xi_1,\xi_2)=\Psi(\xi_1,-\xi_1-\xi_2)$,
\begin{equation}\label{m*2}
\begin{split}
&\|(m^{*2})_j\|_{W^{(s_1',s_2')}_2}
=\|(m(2^j\xi_1, -2^{j}(\xi_1+\xi_2))
\Psi(\xi_1, \xi_2)\|_{W^{(s_1',s_2')}_2}
\\
&\lesssim
\|(m(2^j\xi_1, 2^j\xi_2)
\Psi(\xi_1, -\xi_1-\xi_2)\|_{W^{(s_1'',s_1''+s_2'')}_2}
\lesssim
\sup_{j \in \Z}\|m_j\|_{W^{(s_1,s_2)}_2},
\end{split}
\end{equation}
which gives \eqref{L^2L^2-goal}.

\section{Proof of Theorem \ref{mainthm2}}
\label{section5}

In this section,
we shall consider the $L^2 \times BMO \to L^2$
and $L^2 \times L^2 \to H^1$ boundedness.

\subsection{Pointwise multiplication with homogeneous functions}
\label{section5.1}

We prepare the estimate
for pointwise multiplication with homogeneous functions
in the $L^2$-based Sobolev space,
which will be used in the proof of Theorem \ref{mainthm2}.

\begin{lem}\label{lem-decomposition}
Let $s>n/2$.
If $f \in W^s(\R^n)$ satisfies $f(0)=0$,
then $f$ can be written as $f=\sum_{k \in \N_0}g_k$,
where $g_k(0)=0$,
$\|g_k\|_{L^2} \lesssim 2^{-ks}\|f\|_{W^s}$,
$\mathrm{supp}\, \widehat{g_0} \subset \{|\xi| \le 2\}$,
and
$\mathrm{supp}\, \widehat{g_k} \subset \{2^{k-1} \le |\xi| \le 2^{k+1}\}$,
$k \ge 1$.
\end{lem}
\begin{proof}
Let $\psi_k$, $k \in \N_0$,
be as in the beginning of Section \ref{section3}.
We also use a function $\theta \in \calS(\R^n)$
satisfying $\int_{\R^n}\theta(\xi)\, d\xi=1$
and $\mathrm{supp}\, \theta \subset \{1 \le |\xi| \le 2\}$,
and set $\theta_k(\xi)=2^{-kn}\theta(2^{-k}\xi)$, $k \in \N_0$.
Then
$\mathrm{supp}\, \theta_k
\subset \{2^{k} \le |\xi| \le 2^{k+1}\}$.
Since $f$ belongs to $W^s(\R^n)$
and $s>n/2$,
$\widehat{f}$ is integrable.
Set $a_k=\int_{\R^n}\psi_k(\xi)\widehat{f}(\xi)\, d\xi$.
Since $\sum_{k \in \N_0}\psi_k(\xi)=1$ and
\[
\left|(a_0+a_1+\dots+a_k)\theta_k(\xi)\right|
\lesssim \|\widehat{f}\|_{L^1}2^{-kn}
\to 0
\quad \text{as} \quad k \to \infty
\]
for all $\xi \in \R^n$,
$\widehat{f}$ can be written as
\begin{align*}
\widehat{f}
&=\psi_0 \widehat{f} - a_0 \theta_0
\\
&\qquad
+\psi_1 \widehat{f} + a_0 \theta_0 -(a_0+a_1)\theta_1
\\
&\qquad
+\psi_2 \widehat{f} + (a_0+a_1)\theta_1 - (a_0+a_1+a_2)\theta_2
\\
&\qquad \
\vdots
\\
&\qquad
+\psi_k \widehat{f} + (a_0+a_1+\dots+a_{k-1})\theta_{k-1} - (a_0+a_1+\dots+a_k)\theta_k
\\
&\qquad \
\vdots
\end{align*}
If we set $g_0=\calF^{-1}[\psi_0\widehat{f}-a_0\theta_0]$ and
\[
g_k=\calF^{-1}[\psi_k \widehat{f} + (a_0+a_1+\dots+a_{k-1})\theta_{k-1} - (a_0+a_1+\dots+a_k)\theta_k],
\quad k \ge 1,
\]
we have $f=\sum_{k \in \N_0}g_k$.
Hence, in the rest of the proof,
we shall check that $g_k$, $k \in \N_0$, satisfy the desired conditions.

Using the fact
$\int_{\R^n}\theta_k(\xi)\, d\xi=1$, we see that
\begin{align*}
\int_{\R^n}\widehat{g_k}(\xi)\, d\xi
=a_k+(a_0+a_1+\dots+a_{k-1})-(a_0+a_1+\dots+a_k)=0,
\end{align*}
which implies $g_k(0)=0$.
Since $\langle \xi \rangle \approx 2^k$
for $\xi \in \mathrm{supp}\, \psi_k$,
\[
\|\psi_k\widehat{f}\|_{L^2}
=\|\langle \cdot \rangle^{-s}\langle \cdot \rangle^s
\psi_k\widehat{f}\|_{L^2}
\lesssim 2^{-ks}\|\psi_k\|_{L^{\infty}}
\|\langle \cdot \rangle^s \widehat{f}\|_{L^2}
\lesssim 2^{-ks}\|f\|_{W^s}.
\]
Our assumption $f(0)=0$ implies
$\int_{\R^n}\widehat{f}(\xi)\, d\xi=0$,
and consequently $\sum_{k \in \N_0}a_k=0$.
Thus, since $\sum_{j>k}\psi_j(\xi)=0$ if $|\xi| \le 2^k$,
it follows from Schwarz's inequality that
\begin{align*}
\|(a_0+a_1+\dots+a_k)\theta_{k}\|_{L^2}
&=\left|\sum_{j>k}a_j\right|2^{-kn/2}\|\theta\|_{L^2}
\lesssim
2^{-kn/2}\int_{|\xi| \ge 2^k}
|\widehat{f}(\xi)|\, d\xi
\\
&\le
2^{-kn/2}\left(\int_{|\xi| \ge 2^k}
\langle \xi \rangle^{-2s}\, d\xi\right)^{1/2}
\|\langle \cdot \rangle^s \widehat{f}\|_{L^2}
\approx 2^{-ks}\|f\|_{W^s}.
\end{align*}
Combining these estimates,
we have by Plancherel's theorem
\[
\|g_k\|_{L^2}
\approx \|\widehat{g_k}\|_{L^2}
\lesssim (2^{-ks}+2^{-(k-1)s}+2^{-ks})\|f\|_{W^s}
\approx 2^{-ks}\|f\|_{W^s}.
\]
The support condition of $\widehat{g_k}$
follows from that of $\psi_k$ and $\theta_k$.
The proof is complete.
\end{proof}

\begin{lem}\label{Sobolev-estimate}
Let $s, \widetilde{s} \ge 0$ be such that
$\widetilde{s}<s$,
$\widetilde{s}<N$,
and $s>n/2$,
where $N=[n/2]+1$.
Assume that $\sigma \in C^{N}(\R^n \setminus \{0\})$
is a homogeneous function of degree $0$.
Then
\[
\|\sigma f \|_{W^{\widetilde{s}}}
\lesssim \| f \|_{W^s}
\]
for all $f  \in W^s(\R^n)$
satisfying $f(0)=0$ and
$\mathrm{supp}\, f \subset \{|x| \le 2\}$.
\end{lem}
\begin{proof}
By a change of variables,
we may assume that
$\mathrm{supp}\, f \subset \{|x| \le 1\}$.
Moreover, we can assume that $\|f\|_{W^s}=1$.
Let $\psi \in \calS(\R^n)$
be as in \eqref{LP-homogeneous} with $d=n$.
For $x \in \R^n \setminus \{0\}$
with $|x| \le 1$,
we decompose $\sigma$ as
\[
\sigma(x)
=\sum_{j \in \N_0}
\sigma(x)\psi(2^jx)
=\sum_{j \in \N_0}
\sigma(2^jx)\psi(2^jx)
=\sum_{j \in \N_0}\sigma_j(x),
\]
where we used the homogeneity of $\sigma$
and $\sigma_j(x)=\sigma(2^jx)\psi(2^jx)$.
Then, by Lemma \ref{lem-decomposition},
$\sigma f$ can be written as
\[
\sigma(x)f(x)
=\sum_{j \in \N_0}\sum_{k \in \N_0}
\sigma_j(x)g_k(x),
\]
where $g_k(0)=0$,
$\|g_k\|_{L^2} \lesssim 2^{-ks}$,
$\mathrm{supp}\, \widehat{g_0} \subset \{|\xi| \le 2\}$,
and
$\mathrm{supp}\, \widehat{g_k} \subset \{2^{k-1} \le |\xi| \le 2^{k+1}\}$,
$k \ge 1$.
We shall prove that
\begin{equation}\label{Sobolev-goal-1}
\|\sigma_j g_k\|_{W^{\widetilde{s}}}
\lesssim
\begin{cases}
2^{-k(s-\widetilde{s})},
&j \le k
\\
2^{-j(n/2+1-\widetilde{s})}2^{k(n/2+1-s)},
&j>k .
\end{cases}
\end{equation}
Once this is proved,
we obtain the desired result.
In fact,
since
$n/2+1-\widetilde{s}>0$
and $s-\widetilde{s}>0$,
\begin{align*}
\|\sigma f\|_{W^{\widetilde{s}}}
&\le \left(\sum_{j \le k}+\sum_{j>k} \right)
\|\sigma_j g_k\|_{W^{\widetilde{s}}}
\\
&\lesssim
\sum_{j \le k}
2^{-k(s-\widetilde{s})}
+\sum_{j >k}
2^{-j(n/2+1-\widetilde{s})}2^{k(n/2+1-s)}
\lesssim
\sum_{k \ge 0}
(k+1)2^{-k(s-\widetilde{s})}
<\infty .
\end{align*}

Recall that $N=[n/2]+1$ and $0 \le \widetilde{s}<N$,
and take $0 \le \theta <1$ satisfying
$\widetilde{s}=\theta N$.
Then, in order to prove \eqref{Sobolev-goal-1},
since
\begin{align*}
\|\sigma_j g_k\|_{W^{\widetilde{s}}}
&=\|\langle \cdot \rangle^{\widetilde{s}}
\, \widehat{\sigma_j g_k}\|_{L^2}
\\
&\le \|\widehat{\sigma_j g_k}\|_{L^2}^{1-\theta}
\|\langle \cdot \rangle^N \, \widehat{\sigma_j g_k}\|_{L^2}^\theta
\approx
\|\sigma_j g_k\|_{L^2}^{1-\theta}
\bigg(\sum_{|\alpha| \le N}
\|\partial^{\alpha}(\sigma_j g_k)\|_{L^2}
\bigg)^{\theta},
\end{align*}
it is sufficient to show
\begin{equation}\label{Sobolev-goal-2}
\|\partial^{\alpha}(\sigma_j g_k)\|_{L^2}
\lesssim
\begin{cases}
2^{k|\alpha|}2^{-ks},
&j \le k
\\
2^{j|\alpha|}(2^{-j(n/2+1)}2^{k(n/2+1-s)}),
&j>k
\end{cases}
\end{equation}
for $|\alpha| \le N$.
By Leibniz's rule,
\begin{equation}\label{Leibniz}
\partial^{\alpha}(\sigma_j(x) g_k(x))
=\sum_{\beta \le \alpha}
\binom{\alpha}{\beta}
2^{j|\beta|}
[\partial^{\beta}(\sigma\psi)](2^jx)
\partial^{\alpha-\beta}g_k(x),
\end{equation}
and note that $\sigma \psi$
is not singular at the origin
because of the support condition of $\psi$.
Using the facts
$\mathrm{supp}\, \widehat{g_k}
\subset \{|\xi| \le 2^{k+1}\}$
and $\|g_k\|_{L^2} \lesssim 2^{-ks}$,
$k \ge 0$,
we have
\begin{equation}\label{L2Lq-estimate}
\|\partial^{\gamma}g_k\|_{L^{q}}
\lesssim 2^{k(|\gamma|+(n/2-n/q))}\|g_k\|_{L^2}
\lesssim  2^{k(|\gamma|+(n/2-n/q)-s)}
\end{equation}
for $2 \le q \le \infty$
(see \cite[Remark 1.3.2/1]{Triebel} for the first inequality).
In the case $j \le k$,
it follows from \eqref{L2Lq-estimate} with $q=2$ that
\begin{align*}
\|2^{j|\beta|}
[\partial^{\beta}(\sigma\psi)](2^j \cdot)
\partial^{\alpha-\beta}g_k\|_{L^2}
&\le
2^{j|\beta|}
\|[\partial^{\beta}(\sigma\psi)](2^j \cdot)\|_{L^{\infty}}
\|\partial^{\alpha-\beta}g_k\|_{L^2}
\\
&\lesssim
2^{j|\beta|}
2^{k(|\alpha-\beta|-s)}
\le 2^{k|\beta|}
2^{k(|\alpha-\beta|-s)}=2^{k(|\alpha|-s)}.
\end{align*}
In the case $j>k$,
we consider the two cases
$\beta=\alpha$ and $\beta \neq \alpha$
in the right hand side of \eqref{Leibniz}.
If $\beta=\alpha$,
using the fact $g_k(0)=0$,
we have by \eqref{L2Lq-estimate}
with $q=\infty$
\begin{align*}
\|2^{j|\alpha|}
[\partial^{\alpha}(\sigma\psi)](2^j x)
g_k(x)\|_{L^2_x}
&=2^{j|\alpha|}
\|[\partial^{\alpha}(\sigma\psi)](2^j x)
(g_k(x)-g_k(0))\|_{L^2_x}
\\
&=2^{j|\alpha|}
\Big\|[\partial^{\alpha}(\sigma\psi)](2^j x)
\Big(\int_0^1 x \cdot \nabla g_k(tx)\, dt\Big)\Big\|_{L^2_x}
\\
&\le
2^{j|\alpha|}
\|[\partial^{\alpha}(\sigma\psi)](2^j x)|x|\|_{L^2_x}
\|\nabla g_k\|_{L^\infty}
\\
&\lesssim
2^{j|\alpha|}2^{-j(n/2+1)}2^{k(1+n/2-s)}.
\end{align*}
If $\beta \neq \alpha$,
since $|\alpha-\beta| \ge 1$ and $j>k$,
we have by \eqref{L2Lq-estimate} with $q=\infty$
\begin{align*}
&\|2^{j|\beta|}
[\partial^{\beta}(\sigma\psi)](2^j \cdot)
\partial^{\alpha-\beta}g_k\|_{L^2}
\le
2^{j|\beta|}
\|[\partial^{\beta}(\sigma\psi)](2^j \cdot)\|_{L^{2}}
\|\partial^{\alpha-\beta}g_k\|_{L^\infty}
\\
&\lesssim
2^{j|\beta|}2^{-jn/2}
2^{k(|\alpha-\beta|+n/2-s)}
=2^{j|\beta|}2^{-jn/2}
2^{k(|\alpha-\beta|-1)}
2^{k(1+n/2-s)}
\\
&\le
2^{j|\beta|}2^{-jn/2}
2^{j(|\alpha-\beta|-1)}
2^{k(1+n/2-s)}
=2^{j|\alpha|}2^{-j(n/2+1)}2^{k(1+n/2-s)}.
\end{align*}
Hence, we obtain \eqref{Sobolev-goal-2}.
\end{proof}

\subsection{Proof of Theorem \ref{mainthm2}}
\label{section5.2}

We first consider the boundedness from
$L^2 \times BMO$ to $L^2$.
We use the following decomposition of functions in $BMO$
by Fefferman-Stein \cite[Theorem 3]{Fefferman-Stein}.
For $g\in L^2 \cap BMO$, there exist 
$g_0, g_1, \dots, g_n \in L^2 \cap L^{\infty}$ such that
\begin{equation}\label{BMO-Characterization}
g=g_0+\sum_{k=1}^n R_k(g_k),
\qquad
\sum_{k=0}^n\|g_k\|_{L^{\infty}}
\lesssim \|g\|_{BMO},
\end{equation}
where $R_k$, $1 \le k \le n$, are the Riesz transforms
defined by
\[
R_kg(x)
=\frac{1}{(2\pi)^n}
\int_{\R^n}e^{ix\cdot\xi}
\left(-i\frac{\xi_k}{|\xi|}\right)
\widehat{g}(\xi)\, d\xi
\]
and $\xi=(\xi_1,\dots,\xi_n) \in \R^n$.
As for the assertion to take $g$ and $g_j$ in $L^2$, 
see \cite{Miyachi}.

Let $s_1>0$ and $s_2>n/2$.
Assume that
$\sup_{j \in \Z}\|m_j\|_{W^{(s_1,s_2)}_2}<\infty$
and  $m(\xi_1,0)=0$, $\xi_1 \neq 0$.
By \eqref{BMO-Characterization},
we can write
\begin{align*}
T_m(f_1,f_2)
&=T_m(f_1,f_{2,0})
+\sum_{k=1}^n T_m(f_1,R_k(f_{2,k}))
\\
&=T_{m}(f_1,f_{2,0})
+\sum_{k=1}^n T_{m^{(k)}}(f_1,f_{2,k})
\end{align*}
with
\[
m^{(k)}(\xi_1,\xi_2)
=\left(-i\frac{\xi_{2,k}}{|\xi_2|}\right)
m(\xi_1,\xi_2)
\]
for $f_i \in \calS$, $i=1,2$,
where $\sum_{k=0}^n\|f_{2,k}\|_{L^{\infty}} \lesssim \|f_2\|_{BMO}$.
Then, it follows from Theorem \ref{mainthm} that
\begin{align*}
\|T_m(f_1,f_2)\|_{L^2}
&\lesssim
\left(\sup_{j \in \Z}\|m_j\|_{W^{(s_1,s_2)}_2}\right)
\|f_1\|_{L^2}\|f_{2,0}\|_{L^{\infty}}
\\
&\qquad \qquad 
+\sum_{k=1}^n
\left(\sup_{j \in \Z}
\|(m^{(k)})_j\|_{W^{(s_1,s_2-\epsilon)}_2}\right)
\|f_1\|_{L^2}\|f_{2,k}\|_{L^{\infty}}
\\
&\lesssim
\left(\sup_{j \in \Z}\|m_j\|_{W^{(s_1,s_2)}_2}
+\sum_{k=1}^n\sup_{j \in \Z}
\|(m^{(k)})_j\|_{W^{(s_1,s_2-\epsilon)}_2}\right)
\|f_1\|_{L^2}\|f_2\|_{BMO},
\end{align*}
where $\epsilon$ is a positive number
such that $n/2<s_2-\epsilon<[n/2]+1$,
$m_j$ is defined by \eqref{def-mj},
and $(m^{(k)})_j$ is given by
\begin{align*}
(m^{(k)})_j(\xi_1,\xi_2)
&=m^{(k)}(2^j\xi_1,2^j\xi_2)
\Psi(\xi_1,\xi_2)
\\
&=\left(-i\frac{\xi_{2,k}}{|\xi_2|}\right)
m(2^j\xi_1,2^j\xi_2)\Psi(\xi_1,\xi_2)
=\left(-i\frac{\xi_{2,k}}{|\xi_2|}\right)
m_j(\xi_1,\xi_2).
\end{align*}
Since
$\langle D_{\xi_1} \rangle^{s_1}m_j(\xi_1,0)=0$
and
$\mathrm{supp}\,
\langle D_{\xi_1} \rangle^{s_1}m_j(\xi_1,\cdot)
\subset \{|\xi_2| \le 2\}$ for each $\xi_1$,
we have by Lemma \ref{Sobolev-estimate}
\begin{align*}
\|\langle D_{\xi_1} \rangle^{s_1}
\langle D_{\xi_2} \rangle^{s_2-\epsilon}
(m^{(k)})_j(\xi_1,\xi_2)\|_{L^2_{\xi_2}}
&=\frac{1}{(2\pi)^{n/2}}
\bigg\|\Big(-i\frac{\xi_{2,k}}{|\xi_2|}\Big)
[\langle D_{\xi_1} \rangle^{s_1}
m_j](\xi_1,\xi_2)\bigg\|_{W^{s_2-\epsilon}_{\xi_2}}
\\
&\lesssim
\|\langle D_{\xi_1} \rangle^{s_1}
m_j(\xi_1,\xi_2)\|_{W^{s_2}_{\xi_2}}
\le
\|m_j\|_{W^{(s_1,s_2)}_2}
\end{align*}
for each $\xi_1$.
Hence,
\[
\sup_{j \in \Z}
\|(m^{(k)})_j\|_{W^{(s_1,s_2-\epsilon)}_2}
\lesssim
\sup_{j \in \Z}
\|m_j\|_{W^{(s_1,s_2)}_2},
\]
which gives the first assertion of Theorem \ref{mainthm2}.

We next consider the boundedness from $L^2 \times L^2$ to $H^1$.
Let $s_1>0$ and $s_2>n/2$.
Assume that
$\sup_{j \in \Z}\|m_j\|_{W^{(s_1,s_2)}_2}<\infty$
and  $m(\xi_1,-\xi_1)=0$, $\xi_1 \neq 0$.
We take $s_1'>0$ and $s_2'>n/2$
for the inequality \eqref{m*2} to hold.
Since
\[
m^{*2}(\xi_1,0)
=m(\xi_1,-\xi_1)=0,
\]
where $m^{*2}$ is defined by \eqref{dual-form},
it follows from duality,
the first assertion of Theorem \ref{mainthm2},
and \eqref{m*2} that
\[
\|T_m\|_{L^2 \times L^2 \to H^1}
=\|T_{m^{*2}}\|_{L^2 \times BMO \to L^2}
\lesssim
\sup_{j \in \Z}\|(m^{*2})_j\|_{W^{(s_1', s_2')}_2}
\lesssim
\sup_{j \in \Z}\|m_j\|_{W^{(s_1, s_2)}_2},
\]
which is the second assertion of Theorem \ref{mainthm2}.

\appendix\section{}\label{appendix}

\begin{proof}[Proof of Lemma \ref{TMJ-pointwise}]
By symmetry, it is sufficient to consider the latter inequality.
We write the bilinear operator $T_{m(2^{-j}\cdot)}$ 
as a combination of linear operators in the following form:
\begin{align*}
T_{m(2^{-j}\cdot)}(f_1,f_2)(x)
&=\frac{1}{(2\pi)^n}
\int_{\R^n}e^{ix\cdot\xi_2}\sigma_{f_1,j}(x,2^{-j}\xi_2)
\widehat{f_2}(\xi_2)\, d\xi_2
\\
&=\sigma_{f_1,j}(X,2^{-j}D)f_2(x)
\end{align*}
with
\[
\sigma_{f_1,j}(x,\xi_2)
=\frac{1}{(2\pi)^n}\int_{\R^n}
e^{ix\cdot\xi_1}m(2^{-j}\xi_1,\xi_2)
\widehat{f_1}(\xi_1)\, d\xi_1.
\]
By using the identity
\[
\sigma_{f_1,j}(X,2^{-j}D)f_2(x)
=\int_{\R^n}2^{jn} \calF_2^{-1}\sigma_{f_1,j}(x,2^j(x-y_2))f_2(y_2)\, dy_2,
\]
where $\calF_2^{-1}\sigma_{f_1,j}(x,y_2)$
is the partial inverse Fourier transform
of $\sigma_{f_1,j}(x,\xi_2)$ with respect to the $\xi_2$-variable,
it follows from Schwarz's inequality that
the right hand side of the above is estimated by
\begin{align*}
&\Big( \int_{\R^n}
|(1+2^j|y_2|)^s \calF_2^{-1}\sigma_{f_1,j}(x,2^{j}y_2)|^{2}
2^{jn} dy_2
\Big)^{1/2}
\Big(
\int_{\R^n}\Big|\frac{f_2(y_2)}{(1+2^j|x-y_2|)^s}\Big|^2
2^{jn} dy_2 \Big)^{1/2}
\\
&\approx
\Big( \int_{\R^n}
|\calF_2^{-1}\langle D_{\xi_2} \rangle^{s}
\sigma_{f_1,j}(x,y_2)|^{2} \, dy_2
\Big)^{1/2}
(\zeta_j*|f_2|^2)(x)^{1/2},
\end{align*}
where $\zeta_j(y_2)=2^{jn}(1+2^j|y_2|)^{-2s}$.
Then, Plancherel's theorem gives the desired inequality.
\end{proof}

\begin{proof}[Proof of Lemma \ref{CJM-Carleson}]
We first recall the definition and basic fact of Carleson measures.
A positive measure $\nu$ on $\R^{n+1}_+=\R^n \times (0,\infty)$
is said to be a Carleson measure
if there exists a constant $C>0$ such that
\[
\nu(Q \times (0,\ell(Q))) \le C|Q|
\qquad
\text{for all cubes $Q$ in $\R^n$},
\]
where $\ell(Q)$ is the side length of $Q$.
The infimum of the possible values of the constant $C$ is
called the Carleson constant of $\nu$ and is denoted by $\|\nu\|$.
If $\nu$ is a Carleson measure, then
\[
\int_{\R^{n+1}_+}|F(x,t)|\, d\nu(x,t)
\lesssim \|\nu\|\int_{\R^n}
F^{\ast}(x)\, dx,
\]
where $F^*$ is the nontangential maximal function of $F$,
which is defined by $F^*(x)=\sup_{|x-y|<t}|F(y,t)|$
(see, e.g., \cite[Corollary  3.3.6]{Grafakos-2}).

In \cite[Lemma 3.1]{Grafakos-Miyachi-Tomita},
it was proved that
\begin{equation}\label{Carleson-def}
d\mu=\sum_{j \in \Z}(\zeta_j \ast |\psi(2^{-j}D)b|^2)(x)
\otimes \delta_{2^{-j}}(t)
\end{equation}
is a Carleson measure with Carleson constant
$\|\mu\| \lesssim \|b\|_{BMO}^2$.
Here the meaning of the notation \eqref{Carleson-def}
is that $\mu$ satisfies
\[
\int_{\R^{n+1}_+}F(x,t)\, d\mu(x,t)
=\sum_{j \in \Z}\int_{\R^n}
F(x,2^{-j})(\zeta_j
\ast |\psi(2^{-j}D)b|^2)(x)\, dx
\]
for all nonnegative Borel measurable functions $F$ on $\R^{n+1}_+$.
Hence, by the inequality
$\sup _{|x-y|<2^{-j}}(\zeta_j \ast |f|)(y) \lesssim Mf(x)$,
where $M$ is the Hardy-Littlewood maximal operator
(see, e.g., \cite[Theorem 2.1.10]{Grafakos-1}),
we have
\begin{align*}
&\left(\sum_{j \in \Z}
\int_{\R^n}
(\zeta_j*|f|)(x)^{2}
(\zeta_j * |\psi(2^{-j}D)g|^2)(x)\, dx\right)^{1/2}
\lesssim
\|Mf\|_{L^2}\|g\|_{BMO}.
\end{align*}
Therefore, the boundedness of $M$ on $L^2$
gives the desired result.
\end{proof}


\end{document}